\documentclass[10pt]{article}%
\usepackage{amsmath}
\usepackage{color}
\usepackage{amsfonts}
\usepackage{amssymb}
\usepackage{graphicx}%
\usepackage{mathrsfs}
\usepackage[square, comma, sort&compress, numbers]{natbib}
\usepackage{amssymb, color}
\usepackage[body={15.5cm,21cm}, top=3cm]{geometry}%
\providecommand{\U}[1]{\protect \rule{.1in}{.1in}}

\newtheorem{theorem}{Theorem}[section]

\newtheorem{lemma}[theorem]{Lemma}

\newtheorem{remark}[theorem]{{Remark}}

\newenvironment{proof}[1][Proof]{\noindent \textbf{#1.} }{\  \rule{0.5em}{0.5em}}
\begin{document}
\title{Maximum principle for discrete-time robust stochastic optimal control problem }
\author{Wei He \thanks{Research Center for Mathematics and Interdisciplinary Sciences, Frontiers Science Center for Nonlinear Expectations (Ministry of Education), Shandong University, Qingdao 266237, Shandong, China. hew@sdu.edu.cn. Research supported by China Postdoctoral Science Foundation (2024M761781), the Natural Science Foundation of Shandong Province (ZR2024QA186), the Fundamental Research Funds for the Central Universities.}
}
\date{}
\maketitle
\begin{abstract}
This paper firstly presents the necessary and sufficient conditions for a kind of discrete-time robust stochastic optimal control problem with convex control domains. As it is an ``inf sup problem'', the classical variational method is invalid. We obtain the variational inequality with a common reference probability by systematically using weak convergence approach and the minimax theorem. Moreover, a discrete-time robust investment problem is also studied where the explicit optimal control is given.
\end{abstract}


\section{Introduction}
It is well known that the stochastic maximum principle (SMP in short) is one of the major tools to investigate stochastic optimal control problems which was initially studied by Kushner \cite{K1,K}. The basic idea is to determine the necessary conditions that the optimal control must satisfy. This can be accomplished by solving a forward-backward stochastic system and  minimizing the Hamiltonian function. Since the pioneering works, different versions of the SMP have been developed, adapted to different framework, and relevant results can be found in Bensoussan \cite{BA}, Bismut \cite{B}, Peng \cite{P}, Cadenillas and Karatzas \cite{C}, Yong and Zhou \cite{YZ} and so on.

However, all of the results mentioned above are focused on continuous-time systems. It is important to recognize that discrete-time models are prevalent in nearly all branches of the natural sciences \cite{A, BS}. For example, in engineering applications based on digital hardware, data are available only in discrete time. And motivated by gambling and economic phenomena, the statistical decision theory has received considerable attention in which the state space and action space tend to be finite or countable. What's more, the discrete system can sometimes be looked upon as an approximation to a continuous time system, see Beissner et al. \cite{BL}. Thus, more and more researchers have been devoting their efforts to the topic of  discrete-time optimal control problem. We could refer the interested readers to \cite{H} and \cite{PC} for the discrete-time deterministic maximum principle. As for the stochastic case, the necessary and sufficient conditions for the optimal control were derived in \cite{L,W} and Dong et al. \cite{D} extended the corresponding results to mean-field type discrete-time system. More information about the discrete-time SMP can be found in \cite{HJL, J, S}. There are also several works concerning discrete-time linear-quadratic (LQ) problems. For example, \cite{E} presented the solvability condition for the mean-field LQ problem with finite horizon, and the infinite horizon case was investigated in \cite{N}. Besides, readers can be referred to \cite{BD, R, PW,  Z} for more details.

In control theory, it is commonly assumed that all parameters and the distributions are perfectly known. However, in reality, due to the statistical estimation issues, there always exists ambiguity about the underlying models. One of the best-known way to deal with this ambiguity is the multiple-priors method where the controller selects the optimal policy with respect to the most adverse case (see \cite{BG, GS}). For instance, in Hu and Wang \cite{HW}, $\Gamma$ is a locally compact Polish space which plays the role of parameter space and the random elements $\gamma \in \Gamma$ represents different environment conditions. They assumed that the true law of $\gamma$ is unknown to the controller but belongs to a family of probability distributions $\Lambda$.
Under this model uncertainty setup, they discussed a stochastic recursive optimal robust control problem through SMP approach where the coefficients of the controlled FBSDEs depend on $\gamma \in \Gamma$, i.e.
$$
\left\{\begin{array}{l}
x_\gamma(t)=x_0+\int_0^t b_\gamma\left(s, x_\gamma(s), u(s)\right) d s+\int_0^t \sigma_\gamma\left(s, x_\gamma(s), u(s)\right) d B(s), \\
y_\gamma(t)=\varphi_\gamma\left(x_\gamma(T)\right)+\int_t^T f_\gamma\left(s, x_\gamma(s), y_\gamma(s), z_\gamma(s), u(s)\right) d s-\int_t^T z_\gamma(s) d B(s),
\end{array}\right.
$$
and the cost functional is defined by $J(u)=\sup\limits_{\lambda\in\Lambda} \int_{\Gamma} {y}_{\gamma}(0) \lambda(d\gamma)$. Inspired by this, He et al. \cite{HL} focused on the  maximum principle for a stochastic mean-field type robust control problem under a similar setup and applied the results to study a mean-variance portfolio selection problem with model uncertainty. In order to deal with the recursive optimal control problem with ambiguous volatility, Hu and Ji \cite{HJ} obtained the SMP in the $G$-expectation framework introduced by Peng  \cite{Penga, Pengb, Pengd}. Also noteworthy, Bielecki et al. \cite{BC} firstly developed an adaptive robust control paradigm incorporating the idea of online learning. For other variations of robust control problems, we could refer to \cite{G, HS,  M} and references therein.

In this paper, we aim to obtain the maximum principle for the robust control problem in a discrete-time framework with convex control domain:
\begin{align*}
\left\{\begin{array}{l}
x_{\gamma}(k+1)=b_{\gamma}\left(k, x_{\gamma}(k),  u(k)\right)+\sigma_{\gamma}\left(k, x_{\gamma}(k), u(k)\right)B(k+1) \\
x_{\gamma}(0)=x_{0} \in \mathbb{R}^{n},
\end{array}\right.
\end{align*}
which can be seen as the discretization of a special case in \cite{HW} or the robustness of the problem in \cite{W}. Compared with \cite{HW}, the major obstacle is the lack of It\^{o}'s formula in the discrete-time case. Moreover, unlike the It\^{o}'s integral, the product terms involving the noise like $u(k)B(k+1)$ do not admits zero-mean property. All these will cause trouble to the estimates related to the state equation and adjoint equation. Fortunately, making full use of backward induction, the solution to the adjoint equation can be written out explicitly which together with the independence of $B$ gives the required estimates. The deduction of the duality relation also relies heavily on this explicit solution. In \cite{W} and \cite{D}, to ensure the Hamiltonian system well-defined, they assumed the integrability of $B$ and $u$ depend on the time horizon $N$,
while in the present paper we remove this restriction using a similar argument as in \cite{HJL}. As the problem we focus on is indeed an ``inf sup problem'',  the classical variational method used in \cite{W} and \cite{D}  cannot be directly applied here because of the  subadditivity of the supremum. We draw support from weak convergence approach and the minimax theorem to derive the variational inequality with a common reference probability for any admissible control, the proof of which relies heavily on the compactness of the measure set $\Lambda$.

Our main contributions can be summarized as follows. Firstly, the optimal robust control problem for discrete-time system is formulated. Secondly, based on the variational inequality obtained, the final necessary condition for optimality in terms of the Hamiltonian is derived by using the duality technique together with Fubini theorem. Besides, under additional assumptions, the SMP is also a sufficient condition.  Finally, we apply the SMP to solve a  robust investment problem in which the assets are treated discontinuously in the market and the optimal portfolio is given in an explicit form. In order to get the worst case parameter $\theta^*$, we first let it be underdetermined and then explore through classification discussion.

We shall organize the present paper as follows. We present the basic settings and give some estimates in section 2. In section 3, we establish the necessary and sufficient conditions for the discrete-time optimal control problem under model uncertainty. The last section is devoted to the study of a discrete-time robust investment problem.
\subsubsection*{Notation.}
In this paper, for a given set of parameters $\alpha$, $C(\alpha)$ will denote a positive constant only depending on these parameters and may change from line to line.

Let $N>0$ be a fixed integer, $\mathcal{T}=\{0,1, \ldots, N-1 \}$, $\mathcal{T}^{\prime}=\{1, \ldots, N \}$ and $\mathcal{T}^{\prime\prime}=\{0,1, \ldots, N \}$. On a complete probability space $\left(\Omega, \mathcal{F},P\right)$, we define a sequence of independent random variables $B=\left\{B(k), k \in \mathcal{T}^{\prime}\right\}$ such that $B(k)$ takes value in $\mathbb{R}^d$ for any $k \in \mathcal{T}^{\prime}$ and ${\mathbb{E}}\left[\left|B(k)\right|^{p}\right] \leq C(p)<\infty$ for some $p \geq 2$. Let $\mathcal{F}_{k}  \subset \mathcal{F}$ be the $\sigma$-field generated by $\{B(i), 1\leq i \leq k\}$, i.e., $\mathcal{F}_{0}=\{\emptyset,\Omega\}$ and $\mathcal{F}_k=\sigma\left\{ B(1), \ldots, B(k)\right\}$ for $k=1,\cdots,N$.

Assume $U_0,U_1,\cdots, U_{N-1}$ is a sequence of nonempty convex subsets of $\mathbb{R}^m$. We call $u=\{u(k),k \in \mathcal{T}\}$ an admissible control if $u(k)$ is an $\mathcal{F}_k$-measurable random variable taking value in $U_k$ for each $k \in \mathcal{T}$ and $\mathbb{E} [\sum_{k=0}^{N-1}\left|u(k)\right|^q]<\infty$ with some $q>2$. The set of admissible controls is denoted by $\mathcal{U}$.

\section{Problem formulation}
Estimating market parameters in financial markets is often challenging, and models with significant errors in parameters may be impractical. To address this issue, one possible way is to use the robust model, i.e. to minimize the cost under the worst scenario. This paper contributes to the maximum principle for the robust stochastic control problems in a discrete time framework. Let $(\Gamma,\mathcal{B}(\Gamma))$ be a locally compact Polish space equipped with the distance $\widetilde{d}$. To construct the model uncertainty setup, we use the $\Gamma$-valued random elements $\gamma$ to represent different market conditions and regard $\Lambda$ as the set of all possible probability distributions of $\gamma$. If $\Gamma$ is a singleton set, the problem coincides with the classical one studied in \cite{L,W}.

Our goal is to find optimal control $u^{*} \in \mathcal{U}$, which minimize the following cost functional:
\begin{align}\label{A1}
J(u)=\sup _{\lambda \in \Lambda} \int_{\Gamma}{\mathbb{E}}\left[\sum_{k=0}^{N-1} f_{\gamma}\left(k, x_{\gamma}(k),  u(k)\right)+\phi_{\gamma}\left(x_{\gamma}(N)\right)\right]\lambda(d\gamma),
\end{align}
where $\{x_{\gamma}(k),k \in \mathcal{T}^{\prime\prime}\}$ evolves as follows:
\begin{align}\label{A2}
\left\{\begin{array}{l}
x_{\gamma}(k+1)=b_{\gamma}\left(k, x_{\gamma}(k),  u(k)\right)+\sigma_{\gamma}\left(k, x_{\gamma}(k), u(k)\right)B(k+1) \\
x_{\gamma}(0)=x_{0} \in \mathbb{R}^{n}.
\end{array}\right.
\end{align}
Here, $b_{\gamma}(k, \cdot,  \cdot):\mathbb{R}^n\times U_k\rightarrow
\mathbb{R}^n$, $\sigma_{\gamma}(k, \cdot,  \cdot):\mathbb{R}^n\times U_k\rightarrow
\mathbb{R}^{n\times d}$, $f_{\gamma}(k, \cdot,  \cdot):\mathbb{R}^n\times U_k\rightarrow
\mathbb{R}$, $\phi_{\gamma}:\mathbb{R}^n\rightarrow
\mathbb{R}$ are given continuous functions. We proceed to  introduce some  basic assumptions which will be in force throughout the paper.
\begin{description}
\item[(H1)] For each $k\in \mathcal{T}$ and $\gamma \in \Gamma$, there exists a positive constant $L$ such that $|b_{\gamma}(k,0,0)|+|\sigma_{\gamma}(k,0,0)|\leq L$. In addition, $b_{\gamma}(k,x,u),\sigma_{\gamma}(k,x,u)$ are  continuously differentiable in $(x,u)$ and the derivatives are bounded by $L$
\item[(H2)] For each $k\in \mathcal{T}$ and $\gamma \in \Gamma$, $f_{\gamma}(k,x,u),\phi_{\gamma}(x)$ are continuously differentiable in $(x,u)$ with derivatives bounded by $L(1+|x|+|u|)$, besides $|f_{\gamma}(k,0,0)|\leq L$.
\item[(H3)] There exists a modulus of continuity $\psi:[0,\infty)\rightarrow[0,\infty)$, i.e. a continuous, semiadditive, nondecreasing real function satisfying $\psi(0)=0$,  such that,
   \begin{align*}
   |\chi_{\gamma}(k,x,u)-\chi_{\gamma}(k,x^{\prime},u^{\prime})|\leq \psi(|x-x^{\prime}|+|u-u^{\prime}|),
    \end{align*}
for any $k \in \mathcal{T}, x,x^{\prime} \in \mathbb{R}^n, u,u^{\prime} \in U_k, \gamma \in \Gamma$, where $\chi_{\gamma}$ represents all the derivatives in (H1) and (H2).
\item[(H4)]  For each $ R > 0$, there exists a modulus of continuity $\psi_R:[0,\infty)\rightarrow[0,\infty)$  such that
   \begin{align*}
   |\chi_{\gamma}(k,x,u)-\chi_{\gamma^{\prime}}(k,x,u)|\leq \psi_R(\widetilde{d}(\gamma,\gamma^{\prime})),\ \forall k \in \mathcal{T}, \ |x|,  |u|\leq R,\ \gamma,\gamma^{\prime}\in\Gamma,
    \end{align*}
where $\chi_{\gamma}$ can be $b_{\gamma}(k,x,u),\sigma_{\gamma}(k,x,u),f_{\gamma}(k,x,u),\phi_{\gamma}(x)$ and their derivatives with respect to $(x,u)$.
\item[(H5)] $\Lambda$ is a
weakly compact and convex set of probability measures on $(\Gamma,\mathcal{B}(\Gamma))$.
\end{description}

The following two lemmas reveal that the control system \eqref{A1}-\eqref{A2} is well-defined under above assumptions.
\begin{lemma}\label{myw201}
Suppose (H1) hold. Then for any $u \in \mathcal{U}$, $p\in[2,q]$, the equation \eqref{A2} admits a unique solution $x=\left\{x(k), k \in \mathcal{T}^{\prime\prime}\right\}$  satisfying
\begin{align}\label{A3}
{\mathbb{E}} \left[\sum_{k =0}^N\left|x_{\gamma}(k)\right|^{p}\right] \leq C(L, p,d, N){\mathbb{E}} \left[\left|x_{0}\right|^{p}+\sum_{k=0}^{N-1}\Big(\left|u(k)\right|^{p}+\left|b_{\gamma}\left(k, 0,  0\right)\right|^{p}+\sum_{i=1}^{d}\left|\sigma^i_{\gamma}\left(k, 0,  0\right)\right|^{p}\Big)\right].
\end{align}
Moreover, if we further assume (H4), then $\gamma\rightarrow x_{\gamma}(k)$ is continuous in the following sense
\begin{align}\label{A4}
\lim\limits_{\epsilon\rightarrow 0}\sup\limits_{\widetilde{d}(\gamma,\gamma^{\prime})\leq \epsilon}\mathbb{E}\left[\sum_{k =0}^N|x_{\gamma}(k)-x_{\gamma^{\prime}}(k)|^2\right]=0.
\end{align}
\end{lemma}
\begin{proof}
By assumption (H1) and the independence of $B$, we could get for each $k\in \mathcal{T}$
\begin{align*}
\begin{split}
&{\mathbb{E}} \left[\left|x_{\gamma}(k+1)\right|^{p}\right]\leq C( p){\mathbb{E}} \left[\left|b_{\gamma}\left(k, x_{\gamma}(k),  u(k)\right)\right|^{p}+\sum_{i=1}^{d}\left|\sigma^i_{\gamma}\left(k, x_{\gamma}(k),  u(k)\right)\right|^{p}|B^i(k+1)|^p\right]\\
&\leq C(L,p) {\mathbb{E}} \left[|x_{\gamma}(k)|^p+|u(k)|^p+\left|b_{\gamma}\left(k, 0,  0\right)\right|^{p}+\sum_{i=1}^{d}\Big(|x_{\gamma}(k)|^p+|u(k)|^p+\left|\sigma^i_{\gamma}\left(k, 0,  0\right)\right|^{p}\Big)|B^i(k+1)|^p\right]\\
&\leq C(L,p,d) {\mathbb{E}} \left[|x_{\gamma}(k)|^p+|u(k)|^p+\left|b_{\gamma}\left(k, 0,  0\right)\right|^{p}+\sum_{i=1}^{d}\left|\sigma^i_{\gamma}\left(k, 0,  0\right)\right|^{p}\right].
\end{split}
\end{align*}
Then, \eqref{A3} can be obtained by induction. For notational simplicity, set $\alpha(k)=x_{\gamma}(k)-x_{\gamma^{\prime}}(k)$. From \eqref{A2}, we have
\begin{small}
\begin{align*}
\begin{split}
\alpha(k+1)&=b_{\gamma}\left(k, x_{\gamma}(k),  u(k)\right) -b_{\gamma^{\prime}}\left(k, x_{\gamma^{\prime}}(k), u(k)\right)+\sum_{i=1}^{d} \Big(\sigma^i_{\gamma}\left(k, x_{\gamma}(k),  u(k)\right) -\sigma^i_{\gamma^{\prime}}\left(k, x_{\gamma^{\prime}}(k), u(k)\right)\Big)B^i(k+1)\\
&=\int_{0}^{1}\partial_x b_{\gamma^{\prime}}\left(k,x^{\rho}_{\gamma,\gamma^{\prime}}(k),u(k)\right)d\rho\alpha(k)+\sum_{i=1}^{d} \int_{0}^{1}\partial_x \sigma^i_{\gamma^{\prime}}\left(k,x^{\rho}_{\gamma,\gamma^{\prime}}(k),u(k)\right)d\rho\alpha(k)B^i(k+1)\\
&\ \ \ \ \ \ \ \ \ \ \ \ \ \ \ \ \ \ +{A}_{\gamma,\gamma^{\prime}}^{\gamma}(k)+\sum_{i=1}^{d}{D}_{\gamma,\gamma^{\prime}}^{\gamma,i}(k)B^i(k+1),
\end{split}
\end{align*}
\end{small}
where $x^{\rho}_{\gamma,\gamma^{\prime}}(k)=x_{\gamma^{\prime}}(k)+\rho\left(x_{\gamma}(k)-x_{\gamma^{\prime}}(k)\right)$, ${A}_{\gamma,\gamma^{\prime}}^{\gamma}(k)=b_{\gamma}\left(k, x_{\gamma}(k), u(k)\right) -b_{\gamma^{\prime}}\left(k, x_{\gamma}(k), u(k)\right)$ and ${D}_{\gamma,\gamma^{\prime}}^{\gamma,i}(k)$ is defined similarly. According to the derivation of \eqref{A3}, we deduce that
$$
\mathbb{E}\left[\sum_{k =0}^N|\alpha(k)|^2\right]\leq C(L,d,N) \mathbb{E}\bigg[\sum_{k =0}^{N-1}\Big(|{A}_{\gamma,\gamma^{\prime}}^{\gamma}(k)|^2+\sum_{i=1}^{d} |{D}_{\gamma,\gamma^{\prime}}^{\gamma,i}(k)|^2\Big)\bigg].
$$
In what follows, it remains to prove
\begin{align}\label{A5}
\lim\limits_{\epsilon\rightarrow 0}\sup\limits_{\widetilde{d}(\gamma,\gamma^{\prime})\leq \epsilon}\mathbb{E}\left[\sum_{k =0}^{N-1}\Big(|{A}_{\gamma,\gamma^{\prime}}^{\gamma}(k)|^2+\sum_{i=1}^{d} |{D}_{\gamma,\gamma^{\prime}}^{\gamma,i}(k)|^2\Big)\right]=0.
\end{align}
Assumption (H4) implies that for each $R>0$,
\[
|{A}_{\gamma,\gamma^{\prime}}^{\gamma}(k)|\leq \psi_R(\widetilde{d}(\gamma,\gamma^{\prime}))+C(L)\Big(1+|x_{\gamma}(k)|+|u(k)|\Big)\Big(I_{\{|x_{\gamma}(k)|\geq R\}}+I_{\{|u(k)|\geq R\}}\Big).
\]
Using H{\"o}lder's inequality, Young's inequality together with \eqref{A3}, we obtain
$$
\begin{aligned}
&\mathbb{E}\Big[\sum_{k =0}^{N-1}|x(k)|^2I_{\{|u(k)|\geq R\}}\Big]\leq \mathbb{E}\Big[\sum_{k =0}^{N-1}|x(k)|^{q} \Big]^{\frac{2}{q}} \mathbb{E}\Big[\sum_{k =0}^{N-1} \frac{|u(k)|}{R} \Big]^{\frac{q-2}{q}}\\
&\leq \bigg(\frac{2}{q}\mathbb{E}\Big[\sum_{k =0}^{N-1}|x(k)|^{q}\Big]+\frac{q-2}{q}\mathbb{E}\Big[\sum_{k =0}^{N-1}|u(k)|\Big]\bigg)R^{\frac{2-q}{q}}\leq C(L,N,q,d,x_{0})R^{\frac{2-q}{q}}.
\end{aligned}
$$
Applying the same argument to other terms of ${A}_{\gamma,\gamma^{\prime}}^{\gamma}(k)$ and ${D}_{\gamma,\gamma^{\prime}}^{\gamma,i}(k)$,  we conclude that
$$
\mathbb{E}\left[\sum_{k =0}^{N-1}\Big(|{A}_{\gamma,\gamma^{\prime}}^{\gamma}(k)|^2+\sum_{i=1}^{d} |{D}_{\gamma,\gamma^{\prime}}^{\gamma,i}(k)|^2\Big)\right] \leq C(L,N,q,d,x_{0})\Big(| \psi_R(\widetilde{d}(\gamma,\gamma^{\prime}))|^2+R^{\frac{2-q}{q}}\Big), \ \forall R>0.
$$
Letting $\epsilon\rightarrow 0$ and then $R\rightarrow\infty$ gives \eqref{A5}.
\end{proof}

\begin{lemma}\label{myw202}
Set $y_{\gamma}(0)={\mathbb{E}}\left[\sum_{k=0}^{N-1} f_{\gamma}\left(k, x_{\gamma}(k),  u(k)\right)+\phi_{\gamma}\left(x_{\gamma}(N)\right)\right]$ for simplicity. Under assumptions (H1), (H2) and (H4), $y_\gamma(0) \in C_b(\Gamma)$.
\end{lemma}
\begin{proof}
The boundedness comes immediately from the assumption (H2) and estimate \eqref{A3}. It suffices to prove that the function $\gamma\rightarrow y_{\gamma}(0)$ is continuous. By simple calculation, we could get
\begin{align*}
\begin{split}
|y_{\gamma}(0)-y_{\gamma^{\prime}}(0)|&\leq \mathbb{E}\bigg[\sum_{k=0}^{N-1}\Big(\Big|\int_{0}^{1}\partial_x f_{\gamma^{\prime}}\left(k,x^{\rho}_{\gamma,\gamma^{\prime}}(k),u(k)\right)d\rho\alpha(k)\Big|+|F_{\gamma,\gamma^{\prime}}^{\gamma}(k)|\Big)\\
&\ \ \ +\Big|\int_{0}^{1}\partial_x \phi_{\gamma^{\prime}}\left(x^{\rho}_{\gamma,\gamma^{\prime}}(N)\right)d\rho\alpha(N)\Big|+\Big|\phi_{\gamma}\left(x_{\gamma}(N)\right)-\phi_{\gamma^{\prime}}\left(x_{\gamma}(N)\right)\Big|\bigg],
\end{split}
\end{align*}
where $F_{\gamma,\gamma^{\prime}}^{\gamma}(k)=f_{\gamma}\left(k, x_{\gamma}(k), u(k)\right) -f_{\gamma^{\prime}}\left(k, x_{\gamma}(k), u(k)\right)$. Note that from assumption (H2) we have
$$
\Big|\partial_x f_{\gamma^{\prime}}\left(k,x^{\rho}_{\gamma,\gamma^{\prime}}(k),u(k)\right)\Big|\leq C(L)\Big(1+|x_{\gamma}(k)|+|x_{\gamma^{\prime}}(k)|+|u(k)|\Big).
$$
And with the help of assumption (H4), it holds that
$$
|F_{\gamma,\gamma^{\prime}}^{\gamma}(k)|\leq \psi_R(\widetilde{d}(\gamma,\gamma^{\prime}))+ C(L)\big(1+|x_{\gamma}(k)|^2+|u(k)|^2\big)\big(I_{\{|x_{\gamma}(k)|\geq R\}}+I_{\{|u(k)|\geq R\}}\big).
$$
Thus, by a similar argument as in the proof of Lemma \ref{myw201}, we finally get
\[
|y_{\gamma}(0)-y_{\gamma^{\prime}}(0)|\leq  C(L,N,q,d,x_{0})\Big(\mathbb{E}\Big[\sum_{k =0}^{N}|\alpha(k)|^2\Big]^{\frac{1}{2}}+| \psi_R(\widetilde{d}(\gamma,\gamma^{\prime}))|+R^{\frac{2-q}{q}}\Big).
\]
The desired result follows from \eqref{A4}.
\end{proof}

In conclusion, our aim is to characterize the optimal condition for the following discrete-time robust stochastic control problem
$$
\left\{\begin{array}{cl}
\text { Minimize } & J(u) \\
\text { subject to } & u \in \mathcal{U},
\end{array}\right.
\ \ \ \ \ (\star)$$
which is in fact an ``inf sup problem''. The admissible control ${u}^*=\{u^*(k),k \in \mathcal{T}\}$ solves $(\star)$ is called an optimal control and the corresponding system state is denoted by ${x}_{\gamma}^*=\{x^*_{\gamma}(k),k \in \mathcal{T}^{\prime\prime}\}$.

\section{The maximum principle}
Take ${u}=\{u(k),k \in \mathcal{T}\}$ as an arbitrary element of $ \mathcal{U}$. Since $\{U_k\}_{k=0}^{N-1}$ are convex sets, the perturbed control $u^{\delta}:={u}^*+\delta(u-{u}^*) \in \mathcal{U}$ for any $\delta\in(0,1)$. Denote by ${x}_{\gamma}^{\delta}=\{x^{\delta}_{\gamma}(k),k \in \mathcal{T}^{\prime\prime}\}$ the trajectory associated with $u^{\delta}$. For simplicity, we introduce some short-hand notation
\[
b_{\gamma}^*(k)=b_{\gamma}(k,{x}_{\gamma}^*(k),{u}^*(k)),\  {\phi}_{\gamma}^*(N)=\phi_{\gamma}\left(x_{\gamma}^*(N)\right), \ b_{\gamma}^{\delta}(k)=b_{\gamma}(k,{x}_{\gamma}^{\delta}(k),{u}^{\delta}(k)),
\]
and the other functions and their derivatives can be denoted similarly.
\subsection{Variation along an optimal pair}
In order to derive the first-order necessary condition, we first introduce the variational equation. For each $k \in \mathcal{T}$,
\begin{small}
\begin{align}\label{A6}
\left\{\begin{array}{l}
\bar{x}_{\gamma}(k+1)=\partial_x b^*_{\gamma}(k)\bar{x}_{\gamma}(k)+\partial_u b^*_{\gamma}(k)\hat{u}(k)+\sum\limits_{i=1}^{d}\left\{\partial_x \sigma_{\gamma}^{*,i}(k) \bar{x}_{\gamma}(k)+\partial_u \sigma_{\gamma}^{*,i}(k) \hat{u}(k)\right\} B^{i}(k+1) \\
\bar{x}_{\gamma}(0)=0,
\end{array}\right.
\end{align}
\end{small}
where $\hat{u}(k)=u(k)-u^*(k)$. Due to the derivation of \eqref{A3} and boundedness of the derivatives in (H1), the variational equation \eqref{A6} admits a unique solution $\bar{x}_{\gamma}=\left\{\bar{x}_{\gamma}(k), k=0,1, \ldots, N\right\}$ such that
\begin{align}\label{A7}
{\mathbb{E}} \left[\sum_{k =0}^N\left|\bar{x}_{\gamma}(k)\right|^{p}\right] \leq C(L, p,d, N){\mathbb{E}} \left[\sum_{k=0}^{N-1}\Big(\left|u(k)\right|^{p}+\left|u^*(k)\right|^{p}\Big)\right].
\end{align}
In fact, we could write out the explicit solution to \eqref{A6}. Set
\begin{align}\label{A8}
M_{\gamma}(k)=\partial_x b^*_{\gamma}(k)+\sum\limits_{i=1}^{d}\partial_x \sigma_{\gamma}^{*,i}(k) B^{i}(k+1),\ T_{\gamma}(k)=\Big(\partial_u b^*_{\gamma}(k)+\sum\limits_{i=1}^{d}\partial_u \sigma_{\gamma}^{*,i}(k) B^{i}(k+1)\Big)\hat{u}(k).
\end{align}
Then, \eqref{A6} can be rewritten as $\bar{x}_{\gamma}(k+1)=M_{\gamma}(k)\bar{x}_{\gamma}(k)+T_{\gamma}(k)$.  Since $\bar{x}_{\gamma}(0)=0$, using induction yields for $k=1, \ldots, N$
\begin{align}\label{A9}
\bar{x}_{\gamma}(k)=T_{\gamma}(k-1)+\sum\limits_{i=0}^{k-2}\prod\limits_{j=i+1}^{k-1}M_{\gamma}(j)T_{\gamma}(i),
\end{align}
where we suppose $\sum_{i=0}^{-1}[\cdot]=0$.
\begin{lemma}\label{myw301}
Suppose that (H1)-(H3) are satisfied. Then, we have
\begin{align*}
\lim\limits_{\delta\rightarrow 0}\sup\limits_{\gamma\in\Gamma}\mathbb{E}\left[\sum_{k =0}^N\left|\frac{x^{\delta}_{\gamma}(k)-{x}_{\gamma}^*(k)}{\delta}-\bar{x}_{\gamma}(k)\right|^{2}\right]=0.
\end{align*}
\end{lemma}
\begin{proof}
For $k \in \mathcal{T}$, denote $\widetilde{x}_{\gamma}^{\delta}(k):=\delta^{-1}\left(x_{\gamma}^{\delta}(k)-{x}^*_{\gamma}(k)\right)-\bar{x}_{\gamma}(k)$ for convenience. From \eqref{A2} and \eqref{A6}, we have
\begin{small}
\begin{align}\label{A10}
\begin{split}
\widetilde{x}_{\gamma}^{\delta}(k+1)&=\frac{1}{\delta}\Big(\left(b_{\gamma}^{\delta}(k)-b_{\gamma}^*(k)\right)+\sum_{i=1}^{d}\left(\sigma_{\gamma}^{\delta,i}(k)-\sigma_{\gamma}^{*,i}(k)\right) B^i(k+1)\Big)-\Big(\partial_x b^*_{\gamma}(k)\bar{x}_{\gamma}(k)+\partial_u b^*_{\gamma}(k)\hat{u}(k)\Big)\\
&\ \ \ -\sum_{i=1}^{d}\left(\partial_x \sigma_{\gamma}^{*,i}(k) \bar{x}_{\gamma}(k)+\partial_u \sigma_{\gamma}^{*,i}(k)\hat{u}(k)\right) B^{i}(k+1).
\end{split}
\end{align}
\end{small}
Setting $x_{\gamma}^{\rho,\delta}(k)={x}_{\gamma}^*(k)+\rho\delta(\widetilde{x}_{\gamma}^{\delta}(k)+\bar{x}_{\gamma}(k))$, $u^{\rho,\delta}(k)={u}^*(k)+\rho\delta(u(k)-{u}^*(k))$, we only deal with the diffusion term in detail. Note that
\begin{align*}
\begin{split}
&\frac{1}{\delta}\sum_{i=1}^{d}\left(\sigma_{\gamma}^{\delta,i}(k)-\sigma_{\gamma}^{*,i}(k)\right) B^i(k+1)-\sum_{i=1}^{d}\left(\partial_x \sigma_{\gamma}^{*,i}(k) \bar{x}_{\gamma}(k)+\partial_u \sigma_{\gamma}^{*,i}(k)\big(u(k)-u^*(k)\big)\right)B^{i}(k+1)\\
&=\sum_{i=1}^{d}\int^1_0\partial_x \sigma_{\gamma}^{\rho,\delta,i}(k) \widetilde{x}_{\gamma}^{\delta}(k)B^{i}(k+1)d\rho+\sum_{i=1}^{d}G_{\gamma}^{\delta,i}(k)B^{i}(k+1),
\end{split}
\end{align*}
where $\sigma_{\gamma}^{\rho,\delta,i}(k)=\sigma_{\gamma}^{\delta,i}(k,x_{\gamma}^{\rho,\delta}(k),u^{\rho,\delta}(k))$ and
$$
G_{\gamma}^{\delta,i}(k)=\int^1_0\Big(\partial_x \sigma_{\gamma}^{\rho,\delta,i}(k)-\partial_x \sigma_{\gamma}^{*,i}(k)\Big)\bar{x}_{\gamma}(k)d\rho+\int^1_0\Big(\partial_u \sigma_{\gamma}^{\rho,\delta,i}(k)-\partial_u \sigma_{\gamma}^{*,i}(k)\Big)(u(k)-u^*(k))d\rho.
$$
The drift part can be dealt with in the same way, and denoting $\widetilde{G}_{\gamma}^{\delta}(k)$ similarly as $G_{\gamma}^{\delta}(k)$, \eqref{A10} can be rewritten as
\begin{small}
\begin{align}\label{A1000}
\widetilde{x}_{\gamma}^{\delta}(k+1)=\int^1_0\partial_x b_{\gamma}^{\rho,\delta}(k) \widetilde{x}_{\gamma}^{\delta}(k)d\rho+\widetilde{G}_{\gamma}^{\delta}(k)+\sum_{i=1}^{d}\int^1_0\partial_x \sigma_{\gamma}^{\rho,\delta,i}(k) \widetilde{x}_{\gamma}^{\delta}(k)B^{i}(k+1)d\rho+\sum_{i=1}^{d}G_{\gamma}^{\delta,i}(k)B^{i}(k+1).
\end{align}
\end{small}
Hence, proceeding identically as to derive \eqref{A3}, we obtain the following estimate
$$
\mathbb{E}\left[\sum_{k =0}^N|\widetilde{x}_{\gamma}^{\delta}(k)|^2\right]\leq C(L,d,N) \mathbb{E}\bigg[\sum_{k =0}^{N-1}\Big(|\widetilde{G}_{\gamma}^{\delta}(k)|^2+\sum_{i=1}^{d} |G_{\gamma}^{\delta,i}(k)|^2\Big)\bigg],
$$
which combines with \eqref{A7} and the boundedness of the derivatives in (H1) gives
\begin{align}\label{A11}
{\mathbb{E}} \left[\sum_{k =0}^N\left|\widetilde{x}_{\gamma}^{\delta}(k)\right|^{2}\right] \leq C(L, d, N){\mathbb{E}} \left[\sum_{k=0}^{N-1}\Big(\left|u(k)\right|^{2}+\left|u^*(k)\right|^{2}\Big)\right].
\end{align}
By assumptions (H1) and (H3), we get for each $R>0$
$$
|\widetilde{G}_{\gamma}^{\delta}(k)|^2+\sum_{i=1}^{d} |G_{\gamma}^{\delta,i}(k)|^2\leq C(L,d) \Big(|\bar{x}_{\gamma}(k)|^2+|u(k)-u^*(k)|^2\Big)\Big(\psi(2\delta R)^2+I_{R}\Big),
$$
where $I_R=I_{\{|\widetilde{x}_{\gamma}^{\delta}(k)+\bar{x}_{\gamma}(k)|\geq R\}}+I_{\{|u(k)-{u}^*(k)|\geq R\}}$. By virtue of \eqref{A7} and \eqref{A11} and using the same techniques as in Lemma \ref{myw201}, it holds that
$$
\begin{aligned}
\sup\limits_{\gamma\in\Gamma}\mathbb{E}\bigg[\sum_{k =0}^{N-1}\Big(|\widetilde{G}_{\gamma}^{\delta}(k)|^2+\sum_{i=1}^{d} |G_{\gamma}^{\delta,i}(k)|^2\Big)\bigg]\leq C(L,N,q,d)\Big(\psi(2\delta R)^2+R^{\frac{2-q}{q}}\Big).
\end{aligned}
$$
Sending $\delta\rightarrow 0$ and then  $R\rightarrow\infty$ implies the final conclusion.
\end{proof}

Let us define
\begin{align}\label{A12}
\bar{y}_{\gamma}^u(0)={\mathbb{E}} \left[\sum_{k=0}^{N-1} \Big[\partial_x f^*_{\gamma}(k)\bar{x}_{\gamma}(k)+\partial_u f^*_{\gamma}(k)(u(k)-u^*(k))\Big]+\partial_x \phi^*_{\gamma}(N)\bar{x}_{\gamma}(N)\right],
\end{align}
and
$$
\Lambda^{u^*}=\left\{\lambda\in\Lambda| J({u^*})=\int_{\Gamma} {y}^*_{\gamma}(0)\lambda(d\gamma)\right\}.
$$
It's worth noting that the set $\Lambda^{u^*}$ is not empty. Indeed, for any $\varepsilon_m=\frac{1}{2^m}$, we could find a sequence of $\lambda^m \in \Lambda$ such that
\[
J(u^*)\geq \int_{\Gamma} {y}^*_{\gamma}(0)\lambda^m(d\gamma) \geq J(u^*)-\frac{1}{2^m}.
\]
Since $\Lambda$ is weakly compact, there exists a subsequence $\left\{\lambda^{m_j}\right\}_{j=1}^{\infty}$ converging weakly to some $\lambda_1\in \Lambda$. Then, the following equation comes immediately from the result of Lemma \ref{myw202}
\begin{align}\label{A255}
J(u^*)= \int_{\Gamma} {y}^*_{\gamma}(0)\lambda_1(d\gamma)=\lim _{j \rightarrow \infty}\int_{\Gamma} {y}^*_{\gamma}(0)\lambda^{m_j}(d\gamma),
\end{align}
which implies $\lambda_1 \in \Lambda^{u^*}$.

If $\Gamma$ is a singleton set, from Lemma 2.2 in \cite{W}, the variational inequality can be given directly with respect to $\bar{y}_{\gamma}^u(0)$. However, things get a little more complicated when we consider robust case because of the subadditivity of the supremum. As a consequence, the weak convergence method is needed and the following lemma is indispensable.
\begin{lemma}\label{myw3002}
Under assumptions (H1)-(H4), the function
$\gamma\rightarrow\bar{y}^u_{\gamma}(0)$ is continuous and bounded.
\end{lemma}
\begin{proof}
 The proof will be given in the appendix.
\end{proof}

Now let us display the main contribution of this subsection.
\begin{theorem}\label{myw302}
Let assumptions (H1)-(H5) hold. Then, for each $u \in \mathcal{U}$, there exists a reference probability $\lambda^u\in \Lambda^{u^*}$ such that
\[
\lim\limits_{\delta\rightarrow 0}\frac{J(u^{\delta})-J({u}^*)}{\delta}=\int_{\Gamma} \bar{y}^u_{\gamma}(0) {\lambda}^u(d\gamma)=\sup\limits_{\lambda\in\Lambda^{{u}^*}} \int_{\Gamma} \bar{y}^u_{\gamma}(0) \lambda(d\gamma).
\]

\end{theorem}
\begin{proof}
The proof is based on the idea of Lemma 3.6 in \cite{HW}. We only present a brief overviews here. Denote $
\widetilde{y}_{\gamma}^{\delta}(0):=\delta^{-1}\left(y_{\gamma}^{\delta}(0)- {y}^*_{\gamma}(0)\right)-\bar{y}_{\gamma}^u(0)
$.\\
{\bf Step 1: ($
\lim\limits_{\delta\rightarrow 0}\sup\limits_{\gamma\in\Gamma}|\widetilde{y}^{\delta}_{\gamma}(0)|=0
$)}
We could decompose $\widetilde{y}_{\gamma}^{\delta}(0)$ in a similar way to $\widetilde{x}_{\gamma}^{\delta}(k+1)$ (see \eqref{A1000} in Lemma \ref{myw301}) as follows,
\begin{align}\label{A13}
\widetilde{y}_{\gamma}^{\delta}(0)={\mathbb{E}} \left[\int^1_0\partial_x \phi^{\rho,\delta}_{\gamma}(N)\widetilde{x}_{\gamma}^{\delta}(N)d\rho+\widetilde{C}_{\gamma}^{\delta}(N)+\sum_{k=0}^{N-1}\Big(\int^1_0\partial_x f^{\rho,\delta}_{\gamma}(k)\widetilde{x}_{\gamma}^{\delta}(k)d\rho+{C}_{\gamma}^{\delta}(k)\Big)\right],
\end{align}
where
$$
C_{\gamma}^{\delta}(k)=\int^1_0\Big(\partial_x f_{\gamma}^{\rho,\delta}(k)-\partial_x f_{\gamma}^{*}(k)\Big)\bar{x}_{\gamma}(k)d\rho+\int^1_0\Big(\partial_u f_{\gamma}^{\rho,\delta}(k)-\partial_u f_{\gamma}^{*}(k)\Big)(u(k)-u^*(k))d\rho.
$$
Note that, by assumption (H2) and the fact $
\delta\left(\bar{x}_{\gamma}(k)+\widetilde{x}_{\gamma}^{\delta}(k)\right)=x_{\gamma}^{\delta}(k)-{x}^*_{\gamma}(k)
$, we could get for any $R>0$,
\begin{align*}
\begin{split}
|{C}_{\gamma}^{\delta}(k)|&\leq C(L)\big(|\bar{x}_{\gamma}(k)|+|u(k)-u^*(k)|\big)\Big[\psi(2\delta R)+\big(1+|x_{\gamma}^{\delta}(k)|+|u(k)|+|{x}_{\gamma}^*(k)|+|u^*(k)|\big)I_R\Big].
\end{split}
\end{align*}
Then, a simple calculation gives
\begin{small}
\begin{align*}
\begin{split}
|\widetilde{C}_{\gamma}^{\delta}(N)|+\sum_{k=0}^{N-1}|{C}_{\gamma}^{\delta}(k)|\leq \Big[ 1+\sum_{k=0}^{N} \big(|\bar{x}_{\gamma}(k)|^2+|x_{\gamma}^{\delta}(k)|^2+|{x}_{\gamma}^*(k)|^2\big)+\sum_{k=0}^{N-1}\big(|u^*(k)|^2+|u(k)|^2\big)\Big]\big(\psi(2\delta R)+I_R\big).
\end{split}
\end{align*}
\end{small}
In view of the estimations \eqref{A3}, \eqref{A7} and \eqref{A11}, using similar analysis as in Lemma \ref{myw301}, we could get
$$
\lim\limits_{\delta\rightarrow 0}{\mathbb{E}} \left[|\widetilde{C}_{\gamma}^{\delta}(N)|+\sum_{k=0}^{N-1}|{C}_{\gamma}^{\delta}(k)|\right]=0,
$$
which together with \eqref{A13} and H{\"o}lder's inequality imply
\begin{small}
\begin{align}\label{A133}
\begin{split}
&\lim\limits_{\delta\rightarrow 0}\sup\limits_{\gamma\in\Gamma}|\widetilde{y}^{\delta}_{\gamma}(0)|\leq \lim\limits_{\delta\rightarrow 0}\sup\limits_{\gamma\in\Gamma}{\mathbb{E}} \bigg[\Big(\int^1_0|\partial_x \phi^{\rho,\delta}_{\gamma}(N)|d\rho + \sum_{k=0}^{N-1}\int^1_0|\partial_x f^{\rho,\delta}_{\gamma}(k)|d\rho\Big)\sum_{k=0}^{N} |\widetilde{x}_{\gamma}^{\delta}(k)| +|\widetilde{C}_{\gamma}^{\delta}(N)|+\sum_{k=0}^{N-1}|{C}_{\gamma}^{\delta}(k)|\bigg]\\
&\leq C(L,N) \lim\limits_{\delta\rightarrow 0}\sup\limits_{\gamma\in\Gamma}{\mathbb{E}} \bigg[1+\sum_{k=0}^{N}\Big(|x_{\gamma}^{\delta}(k)|^2+|{x}_{\gamma}^*(k)|^2\Big)+\sum_{k=0}^{N-1}\Big(|u(k)|^2+|u^*(k)|^2\Big)\bigg]^{\frac{1}{2}}\mathbb{E}\Big[\sum_{k=0}^{N}|\widetilde {x}^{\delta}_{\gamma}(k)|^2\Big]^{\frac{1}{2}}.
\end{split}
\end{align}
\end{small}
The desired result comes immediately from the result of Lemma \ref{myw301}.\\
{\bf Step 2: ( The weak convergence method)} For any $\lambda\in \Lambda^{u^*}$, in view of the definition of cost functional \eqref{A1}, we notice that
$$
J(u^{\delta})\geq \int_{\Gamma} y_{\gamma}^{\delta}(0){\lambda}(d\gamma),\ J(u^{*})= \int_{\Gamma} y_{\gamma}^{*}(0){\lambda}(d\gamma),
$$
which implies
\begin{align*}
\frac{J(u^{\delta})-J({u}^*)}{\delta}\geq \int_{\Gamma} \frac{y_{\gamma}^{\delta}(0)-y_{\gamma}^{*}(0)}{\delta} {\lambda}(d\gamma)=
\int_{\Gamma} (\widetilde{y}_{\gamma}^{\delta}(0)+\bar{y}^u_{\gamma}(0) ) {\lambda}(d\gamma).
\end{align*}
It follows from the result of step 1 that $\int_{\Gamma} \widetilde{y}_{\gamma}^{\delta}(0){\lambda}(d\gamma)=o(1)$, i.e. for any $\lambda\in \Lambda^{u^*}$,
\begin{align}\label{A266}
\liminf\limits_{\delta\rightarrow 0}\frac{J(u^{\delta})-J({u}^*)}{\delta}\geq \int_{\Gamma} \bar{y}^u_{\gamma}(0) \lambda(d\gamma).
\end{align}

Next, we will confirm that the reverse inequality is also true. Note that, there exists a sequence $\delta_n \rightarrow 0$ such that
\[
\limsup\limits_{\delta\rightarrow 0}\frac{J(u^{\delta})-J({u}^*)}{\delta}=\lim\limits_{n\rightarrow \infty}\frac{J(u^{\delta_n})-J({u}^*)}{\delta_n}.
\]
Proceeding in a similar manner as to derive \eqref{A255}, we could find $\lambda_{\delta_n}$ satisfying
$$
J(u^{\delta_n})=\int_{\Gamma} y_{\gamma}^{\delta_n}(0) \lambda_{\delta_n}(d\gamma),\ J(u^{*})\geq \int_{\Gamma} y_{\gamma}^{*}(0){\lambda}_{\delta_n}(d\gamma).
$$
Therefore, by using the result of step 1 again, we can demonstrate that
\begin{align*}
\begin{split}
\limsup\limits_{\delta\rightarrow 0}\frac{J(u^{\delta})-J({u}^*)}{\delta}\leq \lim\limits_{n\rightarrow \infty}\int_{\Gamma} (\widetilde{y}_{\gamma}^{\delta_n}(0)+\bar{y}^u_{\gamma}(0) )\lambda_{\delta_n}(d\gamma)=\lim\limits_{n\rightarrow \infty}\int_{\Gamma} \bar{y}^u_{\gamma}(0) {\lambda_{\delta_n}}(d\gamma).
\end{split}
\end{align*}
Due to the weak compactness of $\Lambda$, there exists a probability $\lambda^u$ such that $\lambda_{\delta_n}\stackrel{w}{\rightarrow}\lambda^u$. Then it follows from the result of Lemma \ref{myw3002} that
\begin{align}\label{A277}
\limsup\limits_{\delta\rightarrow 0}\frac{J(u^{\delta})-J({u}^*)}{\delta}\leq  \int_{\Gamma} \bar{y}^u_{\gamma}(0)\lambda^u(d\gamma).
\end{align}
 Furthermore, with the help of the subadditivity of the supremum and the definition of weak convergence, we obtain
\begin{small}
\begin{align*}
\begin{split}
&\Big|J(u^*)-\int_{\Gamma} y_{\gamma}^{*}(0) \lambda^u(d\gamma)\Big|\leq \lim\limits_{n\rightarrow \infty}\bigg(\Big|J(u^*)-J(u^{\delta_n})\Big|+\Big|\int_{\Gamma} y_{\gamma}^{\delta_n}(0) \lambda_{\delta_n}(d\gamma)-\int_{\Gamma} y_{\gamma}^{*}(0) \lambda_{\delta_n}(d\gamma)\Big|\\
&\ \ \ \ \ \ \ \ \ \ \ \ \ +\Big|\int_{\Gamma} y_{\gamma}^{*}(0) \lambda_{\delta_n}(d\gamma)-\int_{\Gamma}y_{\gamma}^{*}(0) \lambda^u(d\gamma)\Big|\bigg)\leq 2\lim\limits_{n\rightarrow \infty} \sup\limits_{\gamma\in\Gamma}\Big|y_{\gamma}^{\delta_n}(0)-y_{\gamma}^{*}(0)\Big|\\
&\leq 2\lim\limits_{n\rightarrow \infty} \sup\limits_{\gamma\in\Gamma}\delta_n\Big(|\widetilde{y}_{\gamma}^{\delta_n}(0)|+|\bar{y}^u_{\gamma}(0)|\Big)
\leq \lim\limits_{n\rightarrow \infty} \sup\limits_{\gamma\in\Gamma}C(L,d,N,x_0)\Big(1+\mathbb{E}\Big[\sum_{k =0}^{N-1}(|u(k)|^2+|u^{*}(k)|^2)\Big]\Big)\delta_n=0.
\end{split}
\end{align*}
\end{small}
where we have used \eqref{A11}, \eqref{A133} and \eqref{AAA1} in the last inequality. That is $\lambda^u \in \Lambda^{u^*}$, which together with \eqref{A266} and \eqref{A277} gives the final result.
\end{proof}

Nevertheless, the necessary condition gotten in Theorem \ref{myw302} is inconvenient for practice, since the reference probability depends on the $u$ given in advance. In what follows, we go one step further with the help of minimax theorem obtaining the variational inequality with a
common reference probability.
\begin{theorem}[Variational inequality]\label{myw303}
Suppose that assumptions (H1)-(H5) hold. Then, there exists a reference probability $\lambda^* \in \Lambda^{u^*}$ such that
\[
\inf\limits_{u\in\mathcal{U}}\int_{\Gamma} \bar{y}^u_{\gamma}(0) {\lambda}^*(d\gamma)\geq 0.
\]
\end{theorem}
\begin{proof}
Firstly, we need to verify $\ell(u,\lambda):=\int_{\Gamma} \bar{y}^u_{\gamma}(0)\lambda(d\gamma)$ satisfies the conditions of the minimax theorem (see Theorem B.1.2 in \cite{Ph}). $\Lambda^{u^*}$ inherits the convexity and weakly compactness of $\Lambda$. And, from \eqref{A6} and \eqref{A12}, it is easy to deduce that
$$
\bar{y}_{\gamma}(0)^{\tau u+(1-\tau)u^{\prime}}=\tau \bar{y}^u_{\gamma}(0)+(1-\tau)\bar{y}^{u^{\prime}}_{\gamma}(0),\ \text{for}\ \tau\in [0,1],\ u,u^{\prime} \in\mathcal{U},
$$
which means $\ell(u,\lambda)$ is convex with respect to $u$. Besides, recalling \eqref{A6}, we have for each $k\in \mathcal{T}$
\begin{align*}
\left\{\begin{array}{l}
\bar{x}^u_{\gamma}(k+1)-\bar{x}^{u^{\prime}}_{\gamma}(k+1)=\partial_x b^*_{\gamma}(k)\Big(\bar{x}^u_{\gamma}(k)-\bar{x}^{u^{\prime}}_{\gamma}(k)\Big)+\partial_u b^*_{\gamma}(k)\Big(u(k)-u^{\prime}(k)\Big)\\
\ \ \ \ \ \ \ \ \ \ \ \ \ +\sum\limits_{i=1}^{d}\left\{\partial_x \sigma_{\gamma}^{*,i}(k) \Big(\bar{x}^u_{\gamma}(k)-\bar{x}^{u^{\prime}}_{\gamma}(k)\Big)+\partial_u \sigma_{\gamma}^{*,i}(k) \Big(u(k)-u^{\prime}(k)\Big)\right\} B^{i}(k+1), \\
\bar{x}^u_{\gamma}(0)-\bar{x}^{u^{\prime}}_{\gamma}(0)=0.
\end{array}\right.
\end{align*}
Following similar steps as to derive \eqref{A3}, we obtain
\[
\mathbb{E}\left[\sum_{k =0}^N|\bar{x}^{u}_{\gamma}(k)-\bar{x}^{u^{\prime}}_{\gamma}(k)|^p\right]\leq C(L,p,d,N) \mathbb{E}\left[\sum_{k =0}^{N-1}|u(k)-u^{\prime}(k)|^p\right].
\]
Combining this with (H2), \eqref{A3}, \eqref{A12} and H{\"o}lder's inequality, we could get the continuity of $\ell$ with respect to $u$, i.e.,
\begin{align*}
\begin{split}
&\Big|\ell(u,\lambda)-\ell(u^{\prime},\lambda)\Big|\leq \sup\limits_{\gamma\in\Gamma}|\bar{y}^u_{\gamma}(0)-\bar{y}^{u^{\prime}}_{\gamma}(0)|\\
&\leq {\mathbb{E}} \bigg[\sum_{k=0}^{N-1}|\partial_u f^*_{\gamma}(k)||u(k)-u^{\prime}(k)|+\sum_{k=0}^{N-1}|\partial_x f^*_{\gamma}(k)||\bar{x}^{u}_{\gamma}(k)-\bar{x}^{u^{\prime}}_{\gamma}(k)|+|\partial_x \phi^*_{\gamma}(N)||\bar{x}^{u}_{\gamma}(N)-\bar{x}^{u^{\prime}}_{\gamma}(N)|\bigg]\\
&\leq C(N)\mathbb{E}\Big[\sum_{k =0}^{N-1}|\partial_x f^*_{\gamma}(k)|^2+|\partial_x \phi^*_{\gamma}(N)|^2\Big]^{\frac{1}{2}}\mathbb{E}\Big[\sum_{k =0}^{N}|\bar{x}^{u}_{\gamma}(k)-\bar{x}^{u^{\prime}}_{\gamma}(k)|^2\Big]^{\frac{1}{2}}\\
&\ \ \ +C(N)\mathbb{E}\Big[\sum_{k =0}^{N-1}|\partial_u f^*_{\gamma}(k)|^2\Big]^{\frac{1}{2}}\mathbb{E}\Big[\sum_{k =0}^{N-1}|u(k)-u^{\prime}(k)|^2\Big]^{\frac{1}{2}}
\leq C(L,N,d,x_0,u^*)\mathbb{E}\Big[\sum_{k =0}^{N-1}|u(k)-u^{\prime}(k)|^2\Big]^{\frac{1}{2}}.
\end{split}
\end{align*}
Therefore, it follows from minimax theorem and the result of Theorem \ref{myw302} that
\[
\inf\limits_{u\in\mathcal{U}}\sup\limits_{\lambda\in\Lambda^{u^*}} \ell(u,\lambda)=\sup\limits_{\lambda\in\Lambda^{u^*}}\inf\limits_{u\in\mathcal{U}}\ell(u,\lambda)\geq 0.
\]

In accordance with the definition of supremum, for any $\varepsilon_m=\frac{1}{2^m}$, there exists a $\lambda^m \in \Lambda^{u^*}$ such that
\[ \inf\limits_{u\in\mathcal{U}}\int_{\Gamma} \bar{y}^u_{\gamma}(0)\lambda^m(d\gamma) \geq \sup\limits_{\lambda\in\Lambda^{u^*}}\inf\limits_{u\in\mathcal{U}}\int_{\Gamma} \bar{y}^u_{\gamma}(0)\lambda(d\gamma)-\frac{1}{2^m}\geq-\frac{1}{2^m} .
\]
Since $\Lambda^{u^*}$ is weakly compact, we could find a subsequence $\left\{\lambda^{m_j}\right\}_{j=1}^{\infty}$ that converges weakly to some $\lambda^* \in \Lambda^{u^*}$. From which it follows that
\begin{align*}
\int_{\Gamma} \bar{y}^u_{\gamma}(0)\lambda^*(d\gamma)=\lim _{j \rightarrow \infty}\int_{\Gamma} \bar{y}^u_{\gamma}(0)\lambda^{m_j}(d\gamma)\geq 0,
\end{align*}
where we have used the fact $\bar{y}^u_{\gamma}(0)\in C_b(\Gamma)$ in Lemma \ref{myw3002}. This completes the proof.
\end{proof}
\subsection{Adjoint equations and necessary conditions for optimality}
Introduce the following adjoint equation which is a discrete time backward stochastic differential equation for $k=0, \ldots, N-2$
\begin{small}
\begin{align}\label{A14}
\begin{cases}
&P_{\gamma}(k)=\mathbb{E}\left[(\partial_x b^*_{\gamma}(k+1))^{\top}P_{\gamma}(k+1)+\sum\limits_{i=1}^d (\partial_x \sigma^{*,i}_{\gamma}(k+1))^{\top}Q^i_{\gamma}(k+1)+(\partial_x f^*_{\gamma}(k+1))^{\top}\mid \mathcal{F}_k\right],\\
&Q_{\gamma}(k)=\mathbb{E}\left[\Big((\partial_x b^*_{\gamma}(k+1))^{\top}P_{\gamma}(k+1)+\sum\limits_{i=1}^d (\partial_x \sigma^{*,i}_{\gamma}(k+1))^{\top}Q^i_{\gamma}(k+1)+(\partial_x f^*_{\gamma}(k+1))^{\top}\Big)B(k+1)^{\top}\mid \mathcal{F}_k\right],\\
&P_{\gamma}(N-1)=\mathbb{E}\left[(\partial_x \phi^*_{\gamma}(N))^{\top}\mid \mathcal{F}_{N-1}\right],\\
&Q_{\gamma}(N-1)=\mathbb{E}\left[(\partial_x \phi^*_{\gamma}(N))^{\top}B(N)^{\top}\mid \mathcal{F}_{N-1}\right].
\end{cases}
\end{align}
\end{small}
Similar to \eqref{A9}, the explicit solution of \eqref{A14} can also be given. Recalling the definition of $M_{\gamma}(k)$ in \eqref{A8} and by backward induction
\begin{small}
\begin{align}\label{A15}
\left\{\begin{array}{l}
P_{\gamma}(k)=\mathbb{E}\left[(\partial_x f^*_{\gamma}(k+1))^{\top}+\Big(\sum\limits_{n=k+2}^N\big(\prod\limits_{j=k+1}^{n-1} M_{\gamma}(j)^{\top}\big) (\partial_x f^*_{\gamma}(n))^{\top}\Big)\mid \mathcal{F}_k\right], \\
Q_{\gamma}(k)=\mathbb{E}\left[(\partial_x f^*_{\gamma}(k+1))^{\top}B(k+1)^{\top}+\Big(\sum\limits_{n=k+2}^N\big(\prod\limits_{j=k+1}^{n-1} M_{\gamma}(j)^{\top}\big) (\partial_x f^*_{\gamma}(n))^{\top}\Big)B(k+1)^{\top} \mid \mathcal{F}_k\right]
\end{array}\right.
\end{align}
\end{small}
satisfies \eqref{A14} for $k=0, \cdots, N-1$ and we suppose $\partial_x f^*_{\gamma}(N)=\partial_x \phi^*_{\gamma}(N)$ and $\sum_{n=N+1}^N[\cdot]=0$. Combining \eqref{A6} and \eqref{A14}, we could see that
\begin{footnotesize}
\begin{align}\label{A16}
\begin{split}
&\mathbb{E}\Big[\partial_x \phi^*_{\gamma}(N)\bar{x}_{\gamma}(N)\mid \mathcal{F}_{N-1}\Big]=\mathbb{E}\bigg[\partial_x \phi^*_{\gamma}(N)\Big[ \partial_x b^*_{\gamma}(N-1)\bar{x}_{\gamma}(N-1)+\partial_u b^*_{\gamma}(N-1)(u(N-1)-u^*(N-1))\\
&\ \ \ + \sum\limits_{i=1}^{d}\left\{\partial_x \sigma_{\gamma}^{*,i}(N-1) \bar{x}_{\gamma}(N-1)+\partial_u \sigma_{\gamma}^{*,i}(N-1) (u(N-1)-u^*(N-1))\right\} B^{i}(N)\Big]\Big| \mathcal{F}_{N-1}\bigg]\\
&= \mathbb{E}\Big[\partial_x \phi^*_{\gamma}(N)\mid \mathcal{F}_{N-1}\Big]\Big( \partial_x b^*_{\gamma}(N-1)\bar{x}_{\gamma}(N-1)+\partial_u b^*_{\gamma}(N-1)(u(N-1)-u^*(N-1))\Big)\\
&\ \ \ +\sum\limits_{i=1}^{d}\mathbb{E}\Big[\partial_x \phi^*_{\gamma}(N)B^{i}(N)\mid \mathcal{F}_{N-1}\Big]\Big(\partial_x \sigma_{\gamma}^{*,i}(N-1) \bar{x}_{\gamma}(N-1)+\partial_u \sigma_{\gamma}^{*,i}(N-1) (u(N-1)-u^*(N-1))  \Big)\\
&=\Big((P_{\gamma}(N-1))^{\top}\partial_x b^*_{\gamma}(N-1)+\sum\limits_{i=1}^d(Q^i_{\gamma}(N-1))^{\top}\partial_x \sigma_{\gamma}^{*,i}(N-1)\Big)\bar{x}_{\gamma}(N-1)\\
&\ \ \ +\Big((P_{\gamma}(N-1))^{\top}\partial_u b^*_{\gamma}(N-1)+\sum\limits_{i=1}^d(Q^i_{\gamma}(N-1))^{\top}\partial_u \sigma_{\gamma}^{*,i}(N-1)\Big)(u(N-1)-u^*(N-1)),
\end{split}
\end{align}
\end{footnotesize}
and moreover for $k=0, \ldots, N-2$
\begin{footnotesize}
\begin{align}\label{A17}
\begin{split}
&\mathbb{E}\bigg[\Big((P_{\gamma}(k+1))^{\top}\partial_x b^*_{\gamma}(k+1)+\sum\limits_{i=1}^d (Q^i_{\gamma}(k+1))^{\top}\partial_x \sigma^{*,i}_{\gamma}(k+1)+\partial_x f^*_{\gamma}(k+1)\Big)\bar{x}_{\gamma}(k+1)\Big| \mathcal{F}_k\bigg]\\
&=\mathbb{E}\bigg[\Delta_{\gamma}(k+1)\Big(\big(\partial_x b^*_{\gamma}(k)+\sum\limits_{i=1}^{d}\partial_x \sigma_{\gamma}^{*,i}(k)B^{i}(k+1)\big)\bar{x}_{\gamma}(k)+\big(\partial_u b^*_{\gamma}(k)+\sum\limits_{i=1}^{d}\partial_u \sigma_{\gamma}^{*,i}(k)B^{i}(k+1)\big)(u(k)-u^*(k))\Big)\Big| \mathcal{F}_k\bigg]\\
&=\Big((P_{\gamma}(k))^{\top}\partial_x b^*_{\gamma}(k)+\sum\limits_{i=1}^d(Q^i_{\gamma}(k))^{\top}\partial_x \sigma_{\gamma}^{*,i}(k)\Big)\bar{x}_{\gamma}(k)+\Big((P_{\gamma}(k))^{\top}\partial_u b^*_{\gamma}(k)+\sum\limits_{i=1}^d(Q^i_{\gamma}(k))^{\top}\partial_u \sigma_{\gamma}^{*,i}(k)\Big)(u(k)-u^*(k)),
\end{split}
\end{align}
\end{footnotesize}
where we denote $\Delta_{\gamma}(k+1)=(P_{\gamma}(k+1))^{\top}\partial_x b^*_{\gamma}(k+1)+\sum\limits_{i=1}^d (Q^i_{\gamma}(k+1))^{\top}\partial_x \sigma^{*,i}_{\gamma}(k+1)+\partial_x f^*_{\gamma}(k+1)$ for simplicity in the second line. Hence, putting \eqref{A16} and \eqref{A17} together gives
\begin{small}
\begin{align}\label{A177}
\begin{split}
&\mathbb{E}\Big[\sum_{k =0}^{N-1}\partial_x f^*_{\gamma}(k)\bar{x}_{\gamma}(k)+\partial_x \phi^*_{\gamma}(N)\bar{x}_{\gamma}(N)\Big]=\mathbb{E}\Big[\sum_{k =0}^{N-1}\Big((P_{\gamma}(k))^{\top}\partial_u b^*_{\gamma}(k)+\sum\limits_{i=1}^d(Q^i_{\gamma}(k))^{\top}\partial_u \sigma_{\gamma}^{*,i}(k)\Big)(u(k)-u^*(k))\Big],
\end{split}
\end{align}
\end{small}
which reveals the duality relationship between $(P_{\gamma},Q_{\gamma})$ and $\bar{x}_{\gamma}$.

Define a Hamiltonian function as below for any $k\in \mathcal{T}$
$$
\mathscr{H}_{\gamma}(k, x, u, P_{\gamma}, Q_{\gamma}):=f_{\gamma} (k, x,  u)+\left\langle P_{\gamma},b_{\gamma}(k, x,  u)\right\rangle +\sum\limits_{i=1}^d\left\langle Q_{\gamma}^i,\sigma^i_{\gamma}(k, x,  u)\right\rangle.
$$
Then, by \eqref{A12} and \eqref{A177}, we could rewrite the variational inequality as
\begin{align}\label{A18}
\inf\limits_{u\in\mathcal{U}}\int_{\Gamma}\mathbb{E}\bigg[\sum_{k=0}^{N-1} \left\langle \partial_u \mathscr{H}_{\gamma}(k,{x}^*_{\gamma}, {u}^*, P_{\gamma}, Q_{\gamma}), (u(k)-{u}^*(k))\right\rangle \bigg]{\lambda}^*(d\gamma)\geq 0.
\end{align}

\begin{remark}\label{myw304}
In \cite{W} and \cite{D}, to ensure the Hamiltonian system well-defined, they assumed the integrability of $B$ and $u$ depend on the time horizon $N$, while in fact, this restriction can be removed by using similar argument as in \cite{HJL}. More precisely, according to  (H2) and \eqref{A3}, we could get
$$
\mathbb{E}\left[|\partial_x f^*_{\gamma}(n)|^2\right]\leq C(L) \mathbb{E}\left[1+|x_{\gamma}^*(n)|^2+|u^*(n)|^2\right]\leq C(L, d, N,x_0).
$$
Moreover, by (H1) and the definition of $M_{\gamma}(j)$ in \eqref{A8}, we have for each $k \in \mathcal{T}$
$$
\big(\prod_{j=k+1}^{n-1} |M_{\gamma}(j)|\big)|B(k+1)|\leq C(L) \prod_{j=k+1}^{n} \big(1+|B(j)|\big).
$$
Then, it follows from H{\"o}lder's inequality that
\begin{small}
\begin{align*}
\begin{split}
&\mathbb{E}\Big[\big|\big(\prod_{j=k+1}^{n-1} M_{\gamma}(j)^{\top}\big) (\partial_x f^*_{\gamma}(n))^{\top}B(k+1)^{\top}\big||u(k)|\Big]\leq \mathbb{E}\Big[|\partial_x f^*_{\gamma}(n)|^2\Big]^{\frac{1}{2}}\mathbb{E}\Big[\prod_{j=k+1}^{n-1} |M_{\gamma}(j)|^2|B(k+1)|^2|u(k)|^2\Big]^{\frac{1}{2}}\\
&\leq C(L, d, N,x_0) \mathbb{E}\Big[\prod_{j=k+1}^{n} \big(1+|B(j)|^2\big)|u(k)|^2\Big]^{\frac{1}{2}}\leq C(L, d, N,x_0) \prod_{j=k+1}^{n}  \mathbb{E}\Big[1+|B(j)|^2\Big]^{\frac{1}{2}}\mathbb{E}\Big[|u(k)|^2\Big]^{\frac{1}{2}}<\infty,
\end{split}
\end{align*}
\end{small}
where the independence of $B$ plays a key role in the last inequality. In view of \eqref{A15} and the fact $u(k)$ is $\mathcal{F}_k$-measurable, we conclude that for any $u \in \mathcal{U}$ and $k \in \mathcal{T}$,
$
\mathbb{E}\left[|P_{\gamma}(k)||u(k)|+|Q_{\gamma}(k)||u(k)|\right]< \infty.
$

\end{remark}

Now, we are in a position to formulate the main result of this paper, the necessary condition for optimality of the problem $(\star)$.
\begin{theorem}[Maximum principle]\label{myw306}
Let assumptions (H1)-(H5) hold and let ${u}^*=\{u^*(k),k \in \mathcal{T}\}$ be an optimal control of problem $(\star)$ with ${x}_{\gamma}^*=\{x^*_{\gamma}(k),k \in \mathcal{T}^{\prime\prime}\}$ being the corresponding optimal state process. Then, there exists a probability measure $\lambda^* \in \Lambda^{u^*}$ and a solution $(P_{\gamma},Q_{\gamma})$ to \eqref{A14} such that for any $u\in U_k$ and $k \in \mathcal{T}$ the following inequality holds,
\begin{align}\label{A19}
\int_{\Gamma} \left\langle \partial_u \mathscr{H}_{\gamma}(k,{x}^*_{\gamma}, {u}^*, P_{\gamma}, Q_{\gamma}),(u-{u}^*(k))\right\rangle {\lambda}^*(d\gamma)\geq 0, \ \text{$d\mathbb{P}$-a.e..}
\end{align}
\end{theorem}
\begin{proof}
It suffices to prove that \eqref{A19} can be derived from \eqref{A18}. For any $k \in \mathcal{T}$, $\mathbb{O}\in \mathcal{F}_{k}$ and $u\in U_k$,  taking $u(i)=u^*(i)$ when $i \neq k$ and $u(k)=uI_{\mathbb{O}}+u^*(k)I_{\mathbb{O}^c}$, then \eqref{A18} becomes
\begin{align}\label{A199}
\inf\limits_{u\in\mathcal{U}}\int_{\Gamma}\mathbb{E}\bigg[\left\langle \partial_u \mathscr{H}_{\gamma}(k,{x}^*_{\gamma}, {u}^*, P_{\gamma}, Q_{\gamma}), (u-{u}^*(k))\right\rangle I_{\mathbb{O}} \bigg]{\lambda}^*(d\gamma)\geq 0.
\end{align}
For shortness of notation, we set $\partial_u \mathscr{H}_{\gamma}^*(k)=\partial_u \mathscr{H}_{\gamma}(k,{x}^*_{\gamma}, {u}^*, P_{\gamma}, Q_{\gamma})$. Then, in order to get \eqref{A19}, we should exchange the order of integration in \eqref{A199}. For this purpose, we need to check that $\partial_u \mathscr{H}_{\gamma}^*(k)$ satisfies the conditions of Fubini theorem.

Firstly, we focus on the integrability. Note that
\begin{align*}
\begin{split}
\int_{\Gamma}\mathbb{E}\left[|\partial_u \mathscr{H}_{\gamma}^*(k)||u-{u}^*(k)|\right]{\lambda}^*(d\gamma)&\leq \sup\limits_{\gamma\in\Gamma}\mathbb{E}\bigg[|\partial_u f^*_{\gamma}(k)||u-{u}^*(k)|+|(P_{\gamma}(k))^{\top}||\partial_u b^*_{\gamma}(k)||u-{u}^*(k)|\\
&\ \ \ \ \ \ +\sum\limits_{i=1}^d|(Q^i_{\gamma}(k))^{\top}||\partial_u \sigma_{\gamma}^{*,i}(k)||u-{u}^*(k)|\bigg].
\end{split}
\end{align*}
According to (H2) and the estimate \eqref{A3}, we get
$$
\mathbb{E}\Big[|\partial_u f^*_{\gamma}(k)||{u}^*(k)|\Big]\leq C(L) \mathbb{E}\Big[1+|{x}^*_{\gamma}(k)|^2+|u^*(k)|^2\Big]^{\frac{1}{2}}\mathbb{E}\Big[|u^*(k)|^2\Big]^{\frac{1}{2}}\leq C(L, d, N,x_0).
$$
Since $\partial_u b^*_{\gamma}(k)$ and $\partial_u \sigma_{\gamma}^{*,i}(k)$ are $\mathcal{F}_k$-measurable and bounded, by a similar argument as in Remark \ref{myw304},  we derive that
$$
\mathbb{E}\bigg[|(P_{\gamma}(k))^{\top}||\partial_u b^*_{\gamma}(k)||{u}^*(k)|+\sum\limits_{i=1}^d|(Q^i_{\gamma}(k))^{\top}||\partial_u \sigma_{\gamma}^{*,i}(k)||{u}^*(k)|\bigg] \leq C(L, d, N,x_0).
$$

As for the measurability of $\partial_u \mathscr{H}_{\gamma}^*(k)$, we could construct a measurable sequence to approximate it following the procedure in the proof of Lemma 3.9 in \cite{HW}. To be specific, since $\Gamma$ is a Polish space, the tightness of $\Lambda$ follows from the Prokhorov theorem. Then, for each fixed $\eta>1$, we could find a compact set $S_{\eta}$ such that $\lambda^*(\gamma \in (S_{\eta})^c)\leq \frac{1}{\eta}$ and a sequence of open balls $\left\{B\left(\gamma_t,\frac{1}{2\eta}\right)\right\}_{t=1}^{T_{\eta}}$ such that $S_{\eta}\subset\cup_{t=1}^{T_{\eta}}B\left(\gamma_t,\frac{1}{2\eta}\right)$. By Partitions of unity theorem (see \cite{R1}), there exists a partition of unity $\{h_1, \cdots, h_{T_{\eta}}\}$ on $S_{\eta}$ subordinate to the cover. Choose $\gamma^*_t$ such that $h_t(\gamma^*_t)>0$ and define
\[
(\partial_u \mathscr{H}_{\gamma}^*)_{\eta}(k)=\sum\limits_{t=1}^{T_{\eta}}\partial_u \mathscr{H}_{\gamma^*_t}^*(k)h_t(\gamma)I_{\{\gamma\in S_{\eta}\}}.
\]
It is easy to verify that for any $k \in \mathcal{T}$, $(\partial_u \mathscr{H}_{\gamma}^*)_{\eta}(k)$ is $\mathcal{B}(\Gamma)\times\mathcal{F}_k$-measurable. It remains to show the convergence result when $\eta\rightarrow \infty$. From the analysis above, the diameter of support $dia(supp(h_t))\leq \frac{1}{\eta}$. Thus, if $h_t(\gamma)\neq 0$, then, $\widetilde{d}(\gamma^*_j,\gamma)\leq \frac{1}{\eta}$. Therefore, by simple calculation, we could get
\begin{small}
\begin{align*}
\begin{split}
&\int_{\Gamma}\mathbb{E}\left[\Big|(\partial_u \mathscr{H}_{\gamma}^*)_{\eta}(k)-\partial_u \mathscr{H}_{\gamma}^*(k)\Big|\right]{\lambda}^*(d\gamma)\leq \int_{\Gamma}\sum\limits_{t=1}^{T_{\eta}}\mathbb{E}\left[\Big|\partial_u \mathscr{H}_{\gamma^*_t}^*(k)-\partial_u \mathscr{H}_{\gamma}^*(k)\Big|\right]h_t(\gamma)I_{\{\gamma\in S_{\eta}\}} {\lambda}^*(d\gamma)\\
&\ \ \ +\int_{\Gamma}\mathbb{E}\left[\Big|\partial_u \mathscr{H}_{\gamma}^*(k)\Big|\right]I_{\{\gamma\in (S_{\eta})^c\}} {\lambda}^*(d\gamma) \leq \sup\limits_{\widetilde{d}(\gamma,\gamma^{\prime})\leq \frac{1}{\eta}}\mathbb{E}\left[\Big|\partial_u \mathscr{H}_{\gamma}^*(k)-\partial_u \mathscr{H}_{\gamma^{\prime}}^*(k)\Big|\right]+C(L, d, N,x_0)\frac{1}{\eta},
\end{split}
\end{align*}
\end{small}
where we have used the fact $\lambda^*(\gamma \in (S_{\eta})^c)\leq \frac{1}{\eta}$ in the last inequality. If we claim
\begin{align}\label{A344}
\lim\limits_{\varepsilon\rightarrow0}\sup\limits_{\widetilde{d}(\gamma,\gamma^{\prime})\leq \varepsilon}\mathbb{E}\bigg[\Big|\partial_u \mathscr{H}_{\gamma}^*(k)-\partial_u \mathscr{H}_{\gamma^{\prime}}^*(k)\Big|\bigg]=0,
\end{align}
which will be proved in the lemma below, then the measurability of $\partial_u \mathscr{H}_{\gamma}^*(k)$ holds.

Consequently, by Fubini theorem, we conclude that for any $k \in \mathcal{T}$, $\mathbb{O}\in \mathcal{F}_{k}$
\begin{align*}
\int_{\Gamma}\mathbb{E}\bigg[\left\langle \partial_u \mathscr{H}^*_{\gamma}(k), (u-{u}^*(k))\right\rangle I_{\mathbb{O}} \bigg]{\lambda}^*(d\gamma)=\mathbb{E}\bigg[\int_{\Gamma}\left\langle \partial_u \mathscr{H}^*_{\gamma}(k), (u-{u}^*(k))\right\rangle{\lambda}^*(d\gamma) I_{\mathbb{O}} \bigg]\geq0,
\end{align*}
which is equivalent to
\begin{align*}
\mathbb{E}\bigg[\int_{\Gamma}\left\langle \partial_u \mathscr{H}^*_{\gamma}(k), (u-{u}^*(k))\right\rangle{\lambda}^*(d\gamma)\mid \mathcal{F}_k \bigg]\geq0.
\end{align*}
Finally, \eqref{A19} comes immediately from the fact that $\int_{\Gamma}\left\langle \partial_u \mathscr{H}^*_{\gamma}(k), (u-{u}^*(k))\right\rangle{\lambda}^*(d\gamma)$ is $\mathcal{F}_k$-measurable.
\end{proof}

It is noteworthy that,  unlike the It\^{o}'s integral, the product terms involving the noise do not admits zero-mean property and this will cause trouble to the proof of \eqref{A344}. The deduction presented in the following lemma relies heavily on the explicit solution of the adjoint equation given in \eqref{A15} and the independence of $B$.
\begin{lemma}
Under assumptions (H1)-(H4), the claim \eqref{A344} holds.
\end{lemma}
\begin{proof}
From the definition of Hamiltonian function, we obtain
\begin{align}\label{AA20}
\begin{split}
\Big|\partial_u \mathscr{H}_{\gamma}^*(k)-\partial_u \mathscr{H}_{\gamma^{\prime}}^*(k)\Big|&\leq C(L)|{P}_{\gamma}(k)-P_{\gamma^{\prime}}(k)|+|P_{\gamma^{\prime}}(k)||\partial_u b^*_{\gamma}(k)-\partial_u b^*_{\gamma^{\prime}}(k)|+|\partial_u f^*_{\gamma}(k)-\partial_u f^*_{\gamma^{\prime}}(k)|\\
&\ \ \ +C(L)|{Q}_{\gamma}(k)-Q_{\gamma^{\prime}}(k)|+|Q_{\gamma^{\prime}}(k)||\partial_u \sigma^*_{\gamma}(k)-\partial_u \sigma^*_{\gamma^{\prime}}(k)|.
\end{split}
\end{align}
By similar analysis as in appendix, we have
\begin{small}
\begin{align}\label{A21}
\begin{split}
\lim\limits_{\varepsilon\rightarrow0}\sup\limits_{\widetilde{d}(\gamma,\gamma^{\prime})\leq \varepsilon}\mathbb{E}\bigg[\sum_{k =0}^{N-1}\Big(|\partial_u b^*_{\gamma}(k)-\partial_u b^*_{\gamma^{\prime}}(k)|^4+|\partial_u \sigma^*_{\gamma}(k)-\partial_u \sigma^*_{\gamma^{\prime}}(k)|^4+|\partial_u f^*_{\gamma}(k)-\partial_u f^*_{\gamma^{\prime}}(k)|^2\Big)\bigg]=0.
\end{split}
\end{align}
\end{small}
Then, we only need to focus on the terms related to ${Q}$. Recalling \eqref{A15}, we could get for any $k \in \mathcal{T}$
\begin{align}\label{A22}
\begin{split}
&\mathbb{E}\bigg[|Q_{\gamma^{\prime}}(k)||\partial_u \sigma^*_{\gamma}(k)-\partial_u \sigma^*_{\gamma^{\prime}}(k)|\bigg]\leq \mathbb{E}\bigg[|\partial_x f^*_{\gamma^{\prime}}(k+1)||B(k+1)||\partial_u \sigma^*_{\gamma}(k)-\partial_u \sigma^*_{\gamma^{\prime}}(k)|\bigg]\\
&\ \ \ +\mathbb{E}\bigg[\Big(\sum\limits_{n=k+2}^N\big(\prod_{j=k+1}^{n-1}| M_{\gamma^{\prime}}(j)|\big) |\partial_x f^*_{\gamma^{\prime}}(n)|\Big)|B(k+1)||\partial_u \sigma^*_{\gamma}(k)-\partial_u \sigma^*_{\gamma^{\prime}}(k)| \bigg].
\end{split}
\end{align}
With the help of \eqref{A21}, H{\"o}lder's inequality together with the independence of $B$, it is easy to verify that the first term tends to zero when $\widetilde{d}(\gamma,\gamma^{\prime}) \rightarrow0$. Similarly, we derive that for any $k \in \mathcal{T}$
\begin{align}\label{A23}
\begin{split}
&\mathbb{E}\bigg[\big(\prod_{j=k+1}^{n-1}| M_{\gamma^{\prime}}(j)|\big) |\partial_x f^*_{\gamma^{\prime}}(n)||B(k+1)||\partial_u \sigma^*_{\gamma}(k)-\partial_u \sigma^*_{\gamma^{\prime}}(k)| \bigg] \\
&\leq \mathbb{E}\Big[|\partial_x f^*_{\gamma^{\prime}}(n)|^2\Big]^{\frac{1}{2}}\mathbb{E}\Big[\prod_{j=k+1}^{n-1} |M_{\gamma}(j)|^2|B(k+1)|^2|\partial_u \sigma^*_{\gamma}(k)-\partial_u \sigma^*_{\gamma^{\prime}}(k)|^2\Big]^{\frac{1}{2}}\\
&\leq C(L, d, N,x_0,u^*) \mathbb{E}\Big[\prod_{j=k+1}^{n} \big(1+|B(j)|^2\big)|\partial_u \sigma^*_{\gamma}(k)-\partial_u \sigma^*_{\gamma^{\prime}}(k)|^2\Big]^{\frac{1}{2}}\\
&\leq C(L, d, N,x_0,u^*) \prod_{j=k+1}^{n}  \mathbb{E}\Big[1+|B(j)|^2\Big]^{\frac{1}{2}}\mathbb{E}\Big[|\partial_u \sigma^*_{\gamma}(k)-\partial_u \sigma^*_{\gamma^{\prime}}(k)|^2\Big]^{\frac{1}{2}},
\end{split}
\end{align}
which implies the second term of \eqref{A22} tends to zero when $\widetilde{d}(\gamma,\gamma^{\prime}) \rightarrow0$. In view of \eqref{A15},
\begin{align}\label{A24}
\begin{split}
&\mathbb{E}\Big[|{Q}_{\gamma}(k)-Q_{\gamma^{\prime}}(k)|\Big]\leq \mathbb{E}\bigg[\Big|\Big((\partial_x f^*_{\gamma}(k+1))^{\top}-(\partial_x f^*_{\gamma^{\prime}}(k+1))^{\top}\Big)B(k+1)^{\top}\Big|\bigg]\\
&\ \ \ +\mathbb{E}\bigg[\Big|\Big(\sum\limits_{n=k+2}^N\big(\prod_{j=k+1}^{n-1} M_{\gamma}(j)^{\top}\big) (\partial_x f^*_{\gamma}(n))^{\top}-\sum\limits_{n=k+2}^N\big(\prod_{j=k+1}^{n-1} M_{\gamma^{\prime}}(j)^{\top}\big) (\partial_x f^*_{\gamma^{\prime}}(n))^{\top}\Big)B(k+1)^{\top}\Big|\bigg].
\end{split}
\end{align}
Proceeding identically as to derive \eqref{A23}, we have for any $k \in \mathcal{T}$
\begin{align}\label{A25}
\begin{split}
&\mathbb{E}\bigg[\Big|\Big(\prod_{j=k+1}^{n-1} M_{\gamma^{\prime}}(j)^{\top} (\partial_x f^*_{\gamma}(n))^{\top}-\prod_{j=k+1}^{n-1} M_{\gamma^{\prime}}(j)^{\top} (\partial_x f^*_{\gamma^{\prime}}(n))^{\top}\Big)B(k+1)^{\top}\Big|\bigg]\\
&=\mathbb{E}\bigg[\prod_{j=k+1}^{n-1}| M_{\gamma^{\prime}}(j)| \big|\partial_x f^*_{\gamma}(n)-\partial_x f^*_{\gamma^{\prime}}(n)\big||B(k+1)|\bigg]\\
&\leq \mathbb{E}\Big[\big|\partial_x f^*_{\gamma}(n)-\partial_x f^*_{\gamma^{\prime}}(n)\big|^2\Big]^{\frac{1}{2}}\mathbb{E}\Big[\prod_{j=k+1}^{n-1} |M_{\gamma}(j)|^2|B(k+1)|^2\Big]^{\frac{1}{2}}\\
&\leq \prod_{j=k+1}^{n}  \mathbb{E}\Big[1+|B(j)|^2\Big]^{\frac{1}{2}}\mathbb{E}\Big[\big|\partial_x f^*_{\gamma}(n)-\partial_x f^*_{\gamma^{\prime}}(n)\big|^2\Big]^{\frac{1}{2}}.
\end{split}
\end{align}
A direct computation yields that
$$
\Big|\prod_{j=k+1}^{n-1} M_{\gamma}(j)^{\top}-\prod_{j=k+1}^{n-1} M_{\gamma^{\prime}}(j)^{\top} \Big|=\Big|\sum_{j =k+1}^{n-1}\Big[(M_{\gamma}(j)^{\top}-M_{\gamma^{\prime}}(j)^{\top})\prod_{i=k+1}^{j-1}M_{\gamma^{\prime}}(i)^{\top}\prod_{l=j+1}^{n-1}M_{\gamma}(l)^{\top}\Big]\Big|,
$$
where $\prod_{i=k+1}^{k}[\cdot]=1$ and $\prod_{l=n}^{n-1}[\cdot]=1$. Using H{\"o}lder's inequality again, we could get
\begin{align}\label{A26}
\begin{split}
&\mathbb{E}\bigg[\big|M_{\gamma}(j)-M_{\gamma^{\prime}}(j)\big|\prod_{i=k+1}^{j-1}|M_{\gamma^{\prime}}(i)|\prod_{l=j+1}^{n-1}|M_{\gamma}(l)||\partial_x f^*_{\gamma}(n)||B(k+1)|\bigg]\\
&\leq \mathbb{E}\Big[|\partial_x f^*_{\gamma}(n)|^2\Big]^{\frac{1}{2}}\mathbb{E}\Big[\big|M_{\gamma}(j)-M_{\gamma^{\prime}}(j)\big|^2\prod_{i=k+1}^{j-1}|M_{\gamma^{\prime}}(i)|^2\prod_{l=j+1}^{n-1}|M_{\gamma}(l)|^2|B(k+1)|^2\Big]^{\frac{1}{2}}\\
&\leq C(L, d, N,x_0,u^*)\mathbb{E}\Big[\big|M_{\gamma}(j)-M_{\gamma^{\prime}}(j)\big|^4\Big]^{\frac{1}{4}}\mathbb{E}\Big[\prod_{i=k+1}^{j-1}|M_{\gamma^{\prime}}(i)|^4\prod_{l=j+1}^{n-1}|M_{\gamma}(l)|^4|B(k+1)|^4\Big]^{\frac{1}{4}}\\
&\leq C(L, d, N,x_0,u^*) \prod_{i=k+1}^{n}  \mathbb{E}\Big[1+|B(i)|^4\Big]^{\frac{1}{4}}\mathbb{E}\Big[\big|M_{\gamma}(j)-M_{\gamma^{\prime}}(j)\big|^4\Big]^{\frac{1}{4}}.
\end{split}
\end{align}
Then, in spirit of the definition of $M_{\gamma}(j)$ in \eqref{A8}, we obtain
\begin{footnotesize}
\begin{align}\label{A27}
\mathbb{E}\Big[\big|M_{\gamma}(j)-M_{\gamma^{\prime}}(j)\big|^4\Big]^{\frac{1}{4}}\leq C(d)\bigg\{ \mathbb{E}\Big[\big|\partial_x b^*_{\gamma}(j)-\partial_x b^*_{\gamma^{\prime}}(j)\big|^4\Big]^{\frac{1}{4}}+\sum\limits_{i=1}^d\mathbb{E}\Big[\big|\partial_x \sigma_{\gamma}^{*,i}(j)-\partial_x \sigma_{\gamma^{\prime}}^{*,i}(j)\big|^4\Big]^{\frac{1}{4}}\mathbb{E}\Big[|B^i(j+1)|^4\Big]^{\frac{1}{4}}\bigg\}.
\end{align}
\end{footnotesize}
Putting \eqref{A24}-\eqref{A27} and \eqref{A21} together, we could see that $\mathbb{E}\Big[|{Q}_{\gamma}(k)-Q_{\gamma^{\prime}}(k)|\Big]$ converges to zero as $\widetilde{d}(\gamma,\gamma^{\prime}) \rightarrow0$. Dealing with the terms related to $P$ in \eqref{AA20} identically, we finally conclude that \eqref{A344} holds.
\end{proof}

\subsection{Sufficient conditions for optimality}
The maximum principle gives a minimum qualification for the candidate optimal solution. It is natural for us to investigate whether the given control is indeed optimal. Thus, in this subsection, let us give the sufficient  conditions for optimality of problem $(\star)$.
\begin{theorem} \label{myw307}
Let (H1)-(H5) hold and suppose that the control ${u}^*=\{u^*(k),k \in \mathcal{T}\}$ and $\lambda^* \in \Lambda^{u^*}$ satisfy
\begin{align*}
\int_{\Gamma}\left\langle \partial_u \mathscr{H}_{\gamma}(k,{x}^*_{\gamma}, {u}^*, P_{\gamma}, Q_{\gamma}), (u-{u}^*(t))\right\rangle {\lambda}^*(d\gamma)\geq 0, \ \forall u\in U_k, \ \text{$d\mathbb{P}$-a.e.,}
\end{align*}
where ${x}^*_{\gamma}$ is the trajectory corresponding to $u^*$ and $( P_{\gamma}, Q_{\gamma})$ is the solution of \eqref{A14}. Further, we assume for any $\gamma \in \Gamma$,
\begin{description}
\item[(H6)] The function $\phi_{\gamma}$ is convex in $x$, and the Hamiltonian $\mathscr{H}_{\gamma}$ is convex with respect to $(x,u)$.
\end{description}
Then, $u^*$ is an optimal control of problem $(\star)$.
\end{theorem}
\begin{proof}
For convenience, we denote $ b_{\gamma}(k)=b_{\gamma}\left(k, x_{\gamma}(k), u(k)\right)$, and similarly for the other functions. Set $\hat{x}_{\gamma}(k)={x}_{\gamma}(k)-{x}^*_{\gamma}(k)$ for any $k \in \mathcal{T}^{\prime\prime}$. Then, from \eqref{A2}, we know that
\begin{align*}
\left\{\begin{array}{l}
\hat{x}_{\gamma}(k+1)=\partial_x b^*_{\gamma}(k)\hat{x}_{\gamma}(k)+\Big(b_{\gamma}(k)-b^*_{\gamma}(k)-\partial_x b^*_{\gamma}(k)\hat{x}_{\gamma}(k)\Big)+\partial_x \sigma^*_{\gamma}(k)\hat{x}_{\gamma}(k)B^i(k+1)\\
\ \ \ \ \ \ \ \ \ \ \ \ \ \ \ \ \ \ +\sum\limits_{i=1}^d\Big(\sigma^i_{\gamma}(k)-\sigma^{*,i}_{\gamma}(k)-\partial_x \sigma^{*,i}_{\gamma}(k)\hat{x}_{\gamma}(k)\Big)B^i(k+1) \\
\hat{x}_{\gamma}(0)=0.
\end{array}\right.
\end{align*}
Noting that $\lambda^* \in \Lambda^{u^*}$, then, for any $u\in\mathcal{U}$, it follows from the convexity of $\phi_{\gamma}$ in (H6) that
\begin{small}
\begin{align}\label{A299}
\begin{split}
J(u)-J(u^*)&\geq \int_{\Gamma}\Big(y_{\gamma}(0)-y_{\gamma}^*(0)\Big){\lambda}^*(d\gamma)=\int_{\Gamma}\mathbb{E}\Big[\sum_{k =0}^{N-1}\Big(f_{\gamma}(k)-f_{\gamma}^*(k)\Big)+\phi_{\gamma}(N)-\phi_{\gamma}^*(N)\Big]{\lambda}^*(d\gamma)\\
&\geq \int_{\Gamma}\mathbb{E}\Big[\sum_{k =0}^{N-1}\Big(f_{\gamma}(k)-f_{\gamma}^*(k)\Big)+\partial_x \phi_{\gamma}^*(N)\Big(x_{\gamma}(N)-x_{\gamma}^*(N)\Big)\Big]{\lambda}^*(d\gamma).
\end{split}
\end{align}
\end{small}
Proceeding identically as to obtain \eqref{A16}, we could derive that
\begin{small}
\begin{align}\label{A29}
\begin{split}
\mathbb{E}\Big[\partial_x \phi^*_{\gamma}(N)\hat{x}_{\gamma}(N)\mid \mathcal{F}_{N-1}\Big]&=\Big((P_{\gamma}(N-1))^{\top}\partial_x b^*_{\gamma}(N-1)+\sum\limits_{i=1}^d(Q^i_{\gamma}(N-1))^{\top}\partial_x \sigma_{\gamma}^{*,i}(N-1)\Big)\hat{x}_{\gamma}(N-1)\\
&+(P_{\gamma}(N-1))^{\top}\Big(b_{\gamma}(N-1)-b^*_{\gamma}(N-1)-\partial_x b^*_{\gamma}(N-1)\hat{x}_{\gamma}(N-1)\Big)\\
&+\sum\limits_{i=1}^d(Q^i_{\gamma}(N-1))^{\top}\Big(\sigma^i_{\gamma}(N-1)-\sigma^{*,i}_{\gamma}(N-1)-\partial_x \sigma^{*,i}_{\gamma}(N-1)\hat{x}_{\gamma}(N-1)\Big),
\end{split}
\end{align}
\end{small}
and moreover $k=0, \ldots, N-2$
\begin{small}
\begin{align}\label{A30}
\begin{split}
&\mathbb{E}\bigg[\partial_x \mathscr{H}_{\gamma}^*(k+1)\hat{x}_{\gamma}(k+1)\Big| \mathcal{F}_k\bigg]=\Big((P_{\gamma}(k))^{\top}\partial_x b^*_{\gamma}(k)+\sum\limits_{i=1}^d(Q^i_{\gamma}(k))^{\top}\partial_x \sigma_{\gamma}^{*,i}(k)\Big)\hat{x}_{\gamma}(k)\\
&\ \ \ +(P_{\gamma}(k))^{\top}\Big(b_{\gamma}(k)-b^*_{\gamma}(k)-\partial_x b^*_{\gamma}(k)\hat{x}_{\gamma}(k)\Big)+\sum\limits_{i=1}^d(Q^i_{\gamma}(k))^{\top}\Big(\sigma^i_{\gamma}(k)-\sigma^{*,i}_{\gamma}(k)-\partial_x \sigma^{*,i}_{\gamma}(k)\hat{x}_{\gamma}(k)\Big).
\end{split}
\end{align}
\end{small}
Combining \eqref{A29} and \eqref{A30} indicates that
\begin{align*}
\begin{split}
\mathbb{E}\Big[\sum_{k =0}^{N-1}\partial_x f^*_{\gamma}(k)\hat{x}_{\gamma}(k)&+\partial_x \phi^*_{\gamma}(N)\hat{x}_{\gamma}(N)\Big]=\mathbb{E}\bigg[\sum_{k =0}^{N-1}\Big[\Big(\mathscr{H}_{\gamma}(k)-\mathscr{H}_{\gamma}^*(k)\Big)-\Big(f_{\gamma}(k)-f_{\gamma}^*(k)\Big)\\
& -\Big((P_{\gamma}(k))^{\top}\partial_x b^*_{\gamma}(k)+\sum\limits_{i=1}^d(Q^i_{\gamma}(k))^{\top}\partial_x \sigma_{\gamma}^{*,i}(k)\Big)\hat{x}_{\gamma}(k)\Big]\bigg].
\end{split}
\end{align*}
That is
\begin{align*}
\begin{split}
\mathbb{E}\Big[\sum_{k =0}^{N-1}\Big(f_{\gamma}(k)-f_{\gamma}^*(k)\Big)+\partial_x \phi^*_{\gamma}(N)\hat{x}_{\gamma}(N)\Big]=\mathbb{E}\bigg[\sum_{k =0}^{N-1}\Big[\Big(\mathscr{H}_{\gamma}(k)-\mathscr{H}_{\gamma}^*(k)\Big)-\partial_x \mathscr{H}_{\gamma}^*(k)\hat{x}_{\gamma}(k)\Big]\bigg],
\end{split}
\end{align*}
which together with \eqref{A299} and the convexity of $\mathscr{H}_{\gamma}$ in (H6) give
\begin{align*}
\begin{split}
J(u)-J(u^*) &\geq \int_{\Gamma} \mathbb{E}\bigg[\sum_{k =0}^{N-1}\Big[\Big(\mathscr{H}_{\gamma}(k)-\mathscr{H}_{\gamma}^*(k)\Big)-\partial_x \mathscr{H}_{\gamma}^*(k)\hat{x}_{\gamma}(k)\Big]\bigg]{\lambda}^*(d\gamma)\\
&\geq \int_{\Gamma}\mathbb{E}\Big[\sum_{k =0}^{N-1}\partial_u \mathscr{H}_{\gamma}^*(k)\big(u(k)-{u}^*(k)\big)dt\Big]{\lambda}^*(d\gamma)\geq 0.
\end{split}
\end{align*}
The proof is complete.
\end{proof}
\section{Discrete-time robust investment problem}
In what follows, we shall apply the maximum principle obtained to deal with a robust investment problem. Assume that there is a market with a bond and $m$ risky stocks treated discontinuously under different market conditions $\gamma\in \Gamma$. Their prices are subject to the following equations for any $k \in \mathcal{T}$:
\begin{align*}
\begin{cases}
&S^0(k+1)-S^0(k)=e(t)S^0(k), \ S^0(0)>0,\\
&S_{\gamma}^i(k+1)-S_{\gamma}^i(k)=\mu^i_{\gamma}(k)S_{\gamma}^i(k)+\sum\limits_{j=1}^{d}\beta^{ij}_{\gamma}(k)S_{\gamma}^i(k)B(k+1),\ S^i_{\gamma}(0)>0,\ i=1,\cdots,m,
\end{cases}
\end{align*}
where ${\mathbb{E}}[B(k+1)\mid \mathcal{F}_{k}]=0$, ${\mathbb{E}}[B(k+1)^2\mid \mathcal{F}_{k}]=1$. The interest rate $e(k)>0$ is a bounded deterministic function and $\mu^i_{\gamma}(k), \beta^{ij}_{\gamma}(k)$ are bounded random variables.  We denote the investor's total wealth at time $k$ by $x_{\gamma}(k)$ and note that it has dynamics given by
\begin{align}\label{A31}
\begin{cases}
&x_{\gamma}(k+1)=\Big(1+e(k)\Big)x_{\gamma}(k)+\Big(\mu_{\gamma}(k)-e(k)1_m\Big)^{\top}u(k)+\sum\limits_{j=1}^{d}(\beta^j_{\gamma}(k))^{\top}u(k)B(k+1)\\
&x_{\gamma}(0)=x(0)>0,
\end{cases}
\end{align}
where $\mu_{\gamma}(k)=(\mu^1_{\gamma}(k),\cdots,\mu^m_{\gamma}(k))^{\top}$, $\beta^j_{\gamma}(k)=(\beta^{1j}_{\gamma}(k),\cdots,\beta^{mj}_{\gamma}(k))^{\top}$, $1_m=(1, \cdots, 1)^{\top}$ and  $u(k)=(u^1(k),\cdots, u^m(k))^{\top}$ stands for a portfolio of the investor  at time $k$. Notably, $\alpha_{\gamma}(t)$ and $\beta^{ij}_{\gamma}(t)$ also should be uniformly continuous with respect to $\gamma$ to ensure (H4) holds. The objective of the investor is to find an admissible portfolio ${u}=\{u(k),k \in \mathcal{T}\}$ which minimizes the difference from a given benchmark $\Psi=\{\Psi(k),k \in \mathcal{T}\}$ and meanwhile maximizes the expected terminal wealth under the worst scenario. That is to solve the following robust single objective optimal control problem:
$$
\left\{\begin{array}{cl}
\text { Minimize } & J(u)=\sup _{\lambda \in \Lambda} \int_{\Gamma}  {\mathbb{E}}\left[\frac{1}{2}\sum_{k=0}^{N-1}\mathbb{G}_{\gamma}(k)\Big(u(k)-\Psi(k)\Big)^2-\mathbb{H}_{\gamma}x_{\gamma}(N) \right] \lambda (d\gamma)\\
\text { subject to } & u \in \mathcal{U} \text { and } (x_{\gamma}(\cdot),u(\cdot)) \text { satisfy } \eqref{A31},
\end{array}\right.
\ \ \ \ \ (\star\star)$$
where $\Psi(k)$ and $\mathbb{G}_{\gamma}(k)\gg 0$ are bounded random variables and $\mathbb{H}_{\gamma}$ is a positive constant. In this case, the Hamiltonian reduces to
\begin{align*}
\mathscr{H}_{\gamma}(k,x,u,P_{\gamma},Q_{\gamma})=\Big(\Big(1+e(k)\Big)x+(\mathbb{A}_{\gamma}(k))^{\top}u\Big)P_{\gamma}+\sum\limits_{j=1}^{d}(\beta^j_{\gamma}(k))^{\top}uQ^j_{\gamma}+\frac{1}{2}\mathbb{G}_{\gamma}(k)\Big(u-\Psi(k)\Big)^2,
\end{align*}
where $\mathbb{A}_{\gamma}(k)=\mu_{\gamma}(k)-e(k)1_m$ and $(P_{\gamma},Q_{\gamma})$ satisfies
\begin{align*}
\begin{cases}
&P_{\gamma}(k)=\mathbb{E}\left[\Big(1+e(k+1)\Big)P_{\gamma}(k+1)\mid \mathcal{F}_k\right],\ P_{\gamma}(N-1)=\mathbb{E}\left[-\mathbb{H}_{\gamma}\mid \mathcal{F}_{N-1}\right],\\
&Q_{\gamma}(k)=\mathbb{E}\left[\Big(1+e(k+1)\Big)P_{\gamma}(k+1)B(k+1)^{\top}\mid \mathcal{F}_k\right],\ Q_{\gamma}(N-1)=\mathbb{E}\left[-\mathbb{H}_{\gamma}B(N)^{\top}\mid \mathcal{F}_{N-1}\right].\\
\end{cases}
\end{align*}

In order to give an explicit expression for the optimal solution to this problem, in what follows, we only focus on the simplest case where $\Gamma=\{1,2\}$. For instance, the stock market is either a bull one ($\gamma=1$) or a bear one ($\gamma=2$) and the coefficients of stock prices are definitely different in these two cases. Suppose that the probability of the market being in a bull one is $\theta$. It is very difficult to predict the market situation, so the specific value of $\theta$ is always unknown. Then, the related probability distributions set is $\Lambda=\{\lambda^{\theta}:\ \theta\in[0,1]\}$, where $\lambda^\theta$ represents a probability measure such that $\lambda^{\theta}(\{1\})=\theta$ and ${\lambda}^\theta(\{2\})=1-\theta$. From Theorem \ref{myw306} and the convexity of $\mathscr{H}_{\gamma}$, there exists a probability measure $\lambda^{\theta^*} \in \Lambda^{u^*}$ such that
\begin{align*}
\int_{\Gamma}\partial_u \mathscr{H}_{\gamma}(k,{x}^*_{\gamma}, {u}^*, P_{\gamma}, Q_{\gamma}) \lambda^{\theta^*}(d\gamma)=0, \ \text{$d\mathbb{P}$-a.e.,}
\end{align*}
that is
\begin{align}\label{A32}
J(u^*)=\sup\limits_{\theta\in[0,1]}\left(\theta y_1^*(0)+(1-\theta)y_2^*(0)\right)=\max(y_1^*(0),y_2^*(0))=\theta^*y_1^*(0)+(1-\theta^*)y_2^*(0),
\end{align}
where $y_{\gamma}(0)={\mathbb{E}}[\frac{1}{2}\sum_{k=0}^{N-1}\mathbb{G}_{\gamma}(k)(u(k)-\Psi(k))^2-\mathbb{H}_{\gamma}x_{\gamma}(N) ] $ and what's more
\begin{align}\label{A33}
\begin{split}
\theta^*&\partial_u \mathscr{H}^*_{1}(k)+(1-\theta^*)\partial_u \mathscr{H}^*_{2}(k)=\theta^*\mathbb{A}_{1}(k)P_1(k)+\theta^*\sum\limits_{j=1}^{d}\beta^j_1(t)Q^j_1(t)+\theta^*\mathbb{G}_{1}(k)\Big(u^*(k)-\Psi(k)\Big)\\
& +(1-\theta^*)\mathbb{A}_{2}(k)P_2(k)+(1-\theta^*)\sum\limits_{j=1}^{d}\beta^j_2(t)Q^j_2(t)+(1-\theta^*)\mathbb{G}_{2}(k)\Big(u^*(k)-\Psi(k)\Big)=0.
\end{split}
\end{align}
Set $\mathbb{G}^{\theta^*}=\theta^*\mathbb{G}_1+(1-\theta^*)\mathbb{G}_2$, $x=\left[
\begin{array}
[c]{cc}%
x_1\\
x_2
\end{array}
\right]$, $\mathbb{A}=\left[
\begin{array}
[c]{cc}%
\mathbb{A}_1 & \mathbb{A}_2
\end{array}
\right]$, $\beta^j=\left[
\begin{array}
[c]{cc}%
\beta^j_1 & \beta^j_2
\end{array}
\right]$, $ \mathbb{H}=\left[
\begin{array}
[c]{cc}%
\mathbb{H}_1\\
\mathbb{H}_2
\end{array}
\right]$,
$
P=\left[
\begin{array}
[c]{cc}%
P_1\\
P_2
\end{array}
\right]$, $Q=\left[
\begin{array}
[c]{cc}%
Q_1\\
Q_2
\end{array}
\right]$ and $\Theta^*=\left[
\begin{array}
[c]{cc}%
\theta^* & 0 \\
0 & (1-\theta^*)
\end{array}
\right]$.
Then \eqref{A33} can be rewritten as
\begin{align}\label{A34}
\mathbb{A}(k)\Theta^*P(k)+\sum\limits_{j=1}^{d}\beta^j(t)\Theta^*Q(k)+\mathbb{G}^{\theta^*}(k)\Big(u^*(k)-\Psi(k)\Big)=0,
\end{align}
where $P(k)$ solves
\begin{align}\label{A35}
\begin{cases}
&P(k)=\mathbb{E}\left[\Big(1+e(k+1)\Big)P(k+1)\mid \mathcal{F}_k\right],\ P(N-1)=\mathbb{E}\left[-\mathbb{H}\mid \mathcal{F}_{N-1}\right],\\
&Q(k)=\mathbb{E}\left[\Big(1+e(k+1)\Big)P(k+1)B(k+1)^{\top}\mid \mathcal{F}_k\right],\ Q(N-1)=\mathbb{E}\left[-\mathbb{H}B(N)^{\top}\mid \mathcal{F}_{N-1}\right].
\end{cases}
\end{align}
And \eqref{A31} turns to
\begin{align}\label{A36}
x(k+1)=\Big(1+e(k)\Big)x(k)+\mathbb{A}(k)^{\top}u(k)+\sum\limits_{j=1}^{d}\beta^j(k)^{\top}u(k)B(k+1),\ k \in \mathcal{T}.
\end{align}
Indeed, we could write out the explicit solution of \eqref{A35}:
$$
P(k)=\begin{cases}-\mathbb{H}\prod\limits_{i=k+1}^{N-1}\left(1+ e(i)\right), & \text { if } 0 \leq k \leq N-2 \\ -\mathbb{H}, & \text { if } k=N-1\end{cases},\ Q(k)=\mathbf{0},\ \forall k \in \mathcal{T},
$$
which together with \eqref{A34} gives
\begin{align}\label{A37}
u^*(k)=\Psi(k)+(\mathbb{G}^{\theta^*}(k))^{-1}\mathbb{A}(k)\Theta^*\mathbb{H}\prod\limits_{i=k+1}^{N-1}\left(1+ e(i)\right),\ k \in \mathcal{T}.
\end{align}
\begin{theorem}\label{myw401}
There exist a constant $\theta^*$ and an admissible control $u^*$ satisfying \eqref{A32} and \eqref{A34}. Moreover, the $u^*$ defined by \eqref{A37} is an optimal portfolio for the robust investment problem $(\star\star)$.

\end{theorem}
\begin{proof}
Let $\theta^*$ be underdetermined first. From the arguments above, \eqref{A34} holds for any $u^*(k,\theta^*)$, $\theta^*\in [0,1]$ defined by \eqref{A37}. Next, we will discuss whether there exists $\theta^*$ satisfying \eqref{A32}  in three situations to further verify the optimality of $\{u^*(k,\theta^*),k \in \mathcal{T}\}$.

{\bf Case 1:} If
$
y^*_1(0,1)\geq y^*_2(0,1),
$
the above robust  investment  problem reduces to a classical one and $(1,u^*(k,1))$ is indeed the optimal solution.

{\bf Case 2:} If $
y^*_1(0,0)\leq y^*_2(0,0),
$ similarly, one could check that $(0,u^*(k,0))$ solves \eqref{A32} and \eqref{A34}, which implies its optimality.

{\bf Case 3:} If $y^*_1(0,1)< y^*_2(0,1)$ and $y^*_1(0,0)> y^*_2(0,0)$, we claim that $\theta^* \mapsto y^*_{\gamma}(0,\theta^*)$  is continuous on $[0,1]$. Then, by the intermediate value theorem, there exists  $\theta^*\in (0,1)$ satisfying
\begin{align*}
y^*_{1}(0,\theta^*)=y^*_{2}(0,\theta^*).
\end{align*}
Consequently, the corresponding $(\theta^*,u^*(k,\theta^*))$ is the desired solution.

In fact, recalling the definition of $y^*_{\gamma}(0,\theta^*)$, we could get
\begin{align*}
\begin{split}
|y^*_{\gamma}(0,\theta_1^*)&-y^*_{\gamma}(0,\theta_2^*)|\leq C(N,\Psi,\mathbb{G}_{\gamma},\mathbb{H}_{\gamma})\bigg(\mathbb{E}\Big[|x^*_{\gamma}(N,\theta_1^*)-x^*_{\gamma}(N,\theta_2^*)|\Big]\\
&+\Big( \mathbb{E}\Big[\sum_{k=0}^{N-1}|u^*(k,\theta_1^*)|^2\Big]^{\frac{1}{2}}+\mathbb{E}\Big[\sum_{k=0}^{N-1}|u^*(k,\theta_2^*)|^2\Big]^{\frac{1}{2}}\Big) \mathbb{E}\Big[\sum_{k=0}^{N-1}|u^*(k,\theta_1^*)-u^*(k,\theta_2^*)|^2\Big]^{\frac{1}{2}} \bigg).
\end{split}
\end{align*}
Moreover, recalling \eqref{A31} and proceeding identically as to derive \eqref{A3}, we obtain
\[
\mathbb{E}\left[|x^*_{\gamma}(N,\theta_1^*)-x^*_{\gamma}(N,\theta_2^*)|\right]\leq C(e,\mu,\beta,N,d) \mathbb{E}\left[\sum_{k =0}^{N-1}|u^*(k,\theta_1^*)-u^*(k,\theta_2^*)|\right].
\]
Thus, it suffices to show the continuity of $\theta^* \mapsto u^*(k,\theta^*)$.  It is easy to verify that $\mathbb{G}^{\theta^*}(k)$ is bounded and continuous with respect to $\theta^*$ and
$$
(\mathbb{G}^{\theta_1^*}(k))^{-1}-(\mathbb{G}^{\theta_2^*}(k))^{-1}=(\mathbb{G}^{\theta_1^*}(k))^{-1}\Big(\mathbb{G}^{\theta_2^*}(k)-\mathbb{G}^{\theta_1^*}(k)\Big)(\mathbb{G}^{\theta_2^*}(k))^{-1}.
$$
which together with $\mathbb{G}_{\gamma}(k)\gg 0$ indicates the boundedness and continuity of $\theta^* \mapsto (\mathbb{G}^{\theta^*}(k))^{-1}$.
By adding and subtracting terms, it follows from \eqref{A37} that for any $k \in \mathcal{T}$,
\begin{align*}
\begin{split}
&|u^*(k,\theta_1^*)-u^*(k,\theta_2^*)|\leq \Big((\mathbb{G}^{\theta_1^*}(k))^{-1}-(\mathbb{G}^{\theta_2^*}(k))^{-1}\Big)\mathbb{A}(k)\Theta_1^*\mathbb{H}\prod\limits_{i=k+1}^{N-1}\left(1+ e(i)\right)\\
&+(\mathbb{G}^{\theta_2^*}(k))^{-1}\mathbb{A}(k)\Big(\Theta_1^*-\Theta_2^*\Big)\mathbb{H}\prod\limits_{i=k+1}^{N-1}\left(1+ e(i)\right)\leq C(\mathbb{A},\mathbb{G},\mathbb{H})|\theta_1^*-\theta_2^*|,
\end{split}
\end{align*}
which implies the claim holds.
\end{proof}

\appendix
\renewcommand\thesection{\normalsize Appendix:  The proof of $\bar{y}^u_{\gamma}(0)\in C_b(\Gamma)$}
\section{ }

\renewcommand\thesection{A}
\begin{proof}[Proof of Lemma \ref{myw3002}]
By assumption \emph{(H2)} and estimations \eqref{A3} and \eqref{A7}, the boundedness is clear. Indeed, making full use of H{\"o}lder's inequality
\begin{align}\label{AAA1}
\begin{split}
|\bar{y}^u_{\gamma}(0)|
&\leq {\mathbb{E}} \bigg[\sum_{k=0}^{N-1}|\partial_u f^*_{\gamma}(k)||u(k)-u^{*}(k)|+\sum_{k=0}^{N-1}|\partial_x f^*_{\gamma}(k)||\bar{x}^{u}_{\gamma}(k)|+|\partial_x \phi^*_{\gamma}(N)||\bar{x}^{u}_{\gamma}(N)|\bigg]\\
&\leq C(N)\mathbb{E}\Big[\sum_{k =0}^{N-1}|\partial_x f^*_{\gamma}(k)|^2+|\partial_x \phi^*_{\gamma}(N)|^2\Big]^{\frac{1}{2}}\mathbb{E}\Big[\sum_{k =0}^{N}|\bar{x}^{u}_{\gamma}(k)|^2\Big]^{\frac{1}{2}}\\
&\ \ \ +C(N)\mathbb{E}\Big[\sum_{k =0}^{N-1}|\partial_u f^*_{\gamma}(k)|^2\Big]^{\frac{1}{2}}\mathbb{E}\Big[\sum_{k =0}^{N-1}(|u(k)|^2+|u^{*}(k)|^2)\Big]^{\frac{1}{2}}\\
&\leq C(L,d,N)\mathbb{E}\Big[1+\sum_{k =0}^{N-1}|u^{*}(k)|^2+\sum_{k =0}^{N}|{x}^*_{\gamma}(k)|^2\Big]^{\frac{1}{2}}\mathbb{E}\Big[\sum_{k =0}^{N-1}(|u(k)|^2+|u^{*}(k)|^2)\Big]^{\frac{1}{2}}\\
& \leq C(L,d,N,x_0)\Big(1+\mathbb{E}\Big[\sum_{k =0}^{N-1}(|u(k)|^2+|u^{*}(k)|^2)\Big]\Big)<\infty.
\end{split}
\end{align}

Then, the following proof of continuity is divided into two steps.\\
{\bf Step 1: ($\gamma\rightarrow\bar{x}_{\gamma}(k)$ is continuous for any $k \in \mathcal{T}^{\prime\prime}$)} Set $\bar{\alpha}(k)=\bar{x}_{\gamma}(k)-\bar{x}_{\gamma^{\prime}}(k)$. From \eqref{A6}, if we denote $
\widetilde{N}_{\gamma,\gamma^{\prime}}(s)=\Big(\partial_x b^*_{\gamma}(k)-\partial_x b^*_{\gamma^{\prime}}(k)\Big)\bar{x}_{\gamma}(k)+\Big(\partial_u b^*_{\gamma}(k)-\partial_u b^*_{\gamma^{\prime}}(k)\Big)\big(u(k)-{u}^*(k)\big)$, and ${{N}}_{\gamma,\gamma^{\prime}}(k)$ similarly for the diffusion term, then we have for any $k \in \mathcal{T}$
\begin{align}\label{AA1}
\begin{split}
\bar{\alpha}(k+1)&=\partial_x b^*_{\gamma^{\prime}}(k)\bar{\alpha}(k)+\widetilde{N}_{\gamma,\gamma^{\prime}}(k)+\sum_{i=1}^{d}\Big(\partial_x \sigma^{*,i}_{\gamma^{\prime}}(k)\bar{\alpha}(k)+{{N}}^i_{\gamma,\gamma^{\prime}}(k)\Big)B^i(k+1).
\end{split}
\end{align}
Using the derivation of \eqref{A3}, it suffices to prove
\begin{align}\label{AA2}
\lim\limits_{\epsilon\rightarrow 0}\sup\limits_{\widetilde{d}(\gamma,\gamma^{\prime})\leq \epsilon}\mathbb{E}\bigg[\sum_{k =0}^{N-1}\Big(|\widetilde{N}_{\gamma,\gamma^{\prime}}(k)|^2+\sum_{i=1}^{d}|{{N}}^i_{\gamma,\gamma^{\prime}}(k)|^2\Big)\bigg]=0
\end{align}
by making full use of assumptions (H3) and (H4). Indeed, taking the first part of $\widetilde{N}_{\gamma,\gamma^{\prime}}(k)$ for example, we note that
$$
\begin{aligned}
\begin{split}
&\mathbb{E}\bigg[\sum_{k =0}^{N-1} |\partial_x b^*_{\gamma}(k)-\partial_x b^*_{\gamma^{\prime}}(k)|^2|\bar{x}_{\gamma}(k)|^2\bigg]\leq 2\mathbb{E}\left[\sum_{k =0}^{N-1} |\partial_x b^*_{\gamma}(k)-\partial_x b_{\gamma^{\prime}}(k,{x}^*_{\gamma}(k),{u}^*(k))|^2|\bar{x}_{\gamma}(k)|^2\right]\\
&\ \ \ \ \ \ \ \ \ \ \ \  \ \ \ \ \ \ \ \ \ +2\mathbb{E}\left[\sum_{k =0}^{N-1} |\partial_xb_{\gamma^{\prime}}(k,{x}^*_{\gamma}(k),{u}^*(k))-\partial_x b^*_{\gamma^{\prime}}(k)|^2|\bar{x}_{\gamma}(k)|^2\right].
\end{split}
\end{aligned}
$$
Then, thanks to assumption (H4) and analyzing similarly as in Lemma \ref{myw201}, we could get
$$
\begin{aligned}
\lim\limits_{\epsilon\rightarrow 0}\sup\limits_{\widetilde{d}(\gamma,\gamma^{\prime})\leq \epsilon}\mathbb{E}\bigg[\sum_{k =0}^{N-1} |\partial_x b^*_{\gamma}(k)-\partial_x b_{\gamma^{\prime}}(k,{x}^*_{\gamma}(k),{u}^*(k))|^2|\bar{x}_{\gamma}(k)|^2\bigg]=0.
\end{aligned}
$$
As for the second part, it follows from (H3) that, for each $\epsilon>0$, we could find a $\xi>0$ such that if $|{x}^*_{\gamma}(k)-{x}^*_{\gamma^{\prime}}(k)|\leq \xi$, then $\Big|\partial_x b_{\gamma^{\prime}}({x}^*_{\gamma}(k),{u}^*(k))-\partial_x b^*_{\gamma^{\prime}}(k)\Big|\leq \epsilon$, that is
\begin{align}\label{AA3}
\Big|\partial_x b_{\gamma^{\prime}}(k,{x}^*_{\gamma}(k),{u}^*(k))-\partial_x b^*_{\gamma^{\prime}}(k)\Big|\leq \epsilon+C(L)I_{\{|{x}^*_{\gamma}(k)-{x}^*_{\gamma^{\prime}}(k)|\geq \xi\}}.
\end{align}
Thus, by the result of Lemma \ref{myw201} and the estimation \eqref{A7}, we derive that,
$$
\begin{aligned}
\begin{split}
&\lim\limits_{\epsilon\rightarrow 0}\sup\limits_{\widetilde{d}(\gamma,\gamma^{\prime})\leq \epsilon}\mathbb{E}\bigg[\sum_{k =0}^{N-1} |\partial_xb_{\gamma^{\prime}}(k,{x}^*_{\gamma}(k),{u}^*(k))-\partial_x b^*_{\gamma^{\prime}}(k)|^2|\bar{x}_{\gamma}(k)|^2\bigg]\\
&\ \ \ \ \leq  C(L,N,d,q)\lim\limits_{\epsilon\rightarrow 0}\bigg(|\epsilon|^2+\xi^{\frac{2-q}{q}}\sup\limits_{\widetilde{d}(\gamma,\gamma^{\prime})\leq \epsilon}{\mathbb{E}\Big[\sum_{k =0}^{N-1} |{x}^*_{\gamma}(k)-{x}^*_{\gamma^{\prime}}(k)|\Big]^{\frac{q-2}{q}}}\bigg)=0.
\end{split}
\end{aligned}
$$
Treating other terms in a same way, we could conclude that \eqref{AA2} holds.\\
{\bf Step 2: (The continuity of $\gamma\rightarrow\bar{y}^u_{\gamma}(0)$ )} By \eqref{A12}, it is clear that
\begin{align}\label{AA4}
\begin{split}
\bar{y}^u_{\gamma}(0)-\bar{y}^u_{\gamma^{\prime}}(0)
&=\mathbb{E}\bigg[\partial_x \phi^*_{\gamma^{\prime}}(N)\bar{\alpha}(N)+\sum_{k =0}^{N-1} \partial_x f^*_{\gamma^{\prime}}(k)\bar{\alpha}(k)+{{K}}_{\gamma,\gamma^{\prime}}\bigg],
\end{split}
\end{align}
where ${{K}}_{\gamma,\gamma^{\prime}}=\sum_{k =0}^{N-1} \Big[\Big(\partial_x f^*_{\gamma}(k)-\partial_x f^*_{\gamma^{\prime}}(k)\Big)\bar{x}_{\gamma}(k)+\Big(\partial_u f^*_{\gamma}(k)-\partial_u f^*_{\gamma^{\prime}}(k)\Big)\big(u(k)-{u}^*(k)\big)\Big]+\Big(\partial_x \phi^*_{\gamma}(N)-\partial_x \phi^*_{\gamma^{\prime}}(N)\Big)\bar{x}_{\gamma}(N)$. Observe that, for any $R>0$,
\begin{small}
\begin{align*}
\begin{split}
\Big|\partial_x f^*_{\gamma}(k)-\partial_x f_{\gamma^{\prime}}(k,{x}^*_{\gamma}(k),{u}^*(k))\Big|\leq \psi_R(\widetilde{d}(\gamma,\gamma^{\prime}))+ C(L) \big(1+|x_{\gamma}^*(k)|+|u^*(k)|\big)\big(I_{\{|x_{\gamma}^*(k)|\geq R\}}+I_{\{|u^*(k)|\geq R\}}\big).
\end{split}
\end{align*}
\end{small}
Furthermore, like \eqref{AA3}, we have
\begin{align*}
\begin{split}
\Big|\partial_x f_{\gamma^{\prime}}(k,{x}^*_{\gamma}(k),{u}^*(k))-\partial_x f^*_{\gamma^{\prime}}(k)\Big|\leq \epsilon+ C(L)\big(1+|x_{\gamma}^*(k)|+|u^*(k)|+|x_{\gamma^{\prime}}^*(k)|\big)I_{\{|{x}^*_{\gamma}(k)-{x}^*_{\gamma^{\prime}}(k)|\geq \xi\}}.
\end{split}
\end{align*}
Combining the above two inequalities shows that
\begin{align*}
\begin{split}
&\mathbb{E}\Big[\sum_{k =0}^{N-1} \Big|\Big(\partial_x f^*_{\gamma}(k)-\partial_x f^*_{\gamma^{\prime}}(k)\Big)\bar{x}_{\gamma}(k)\Big|\Big]\leq \Big(\psi_R(\widetilde{d}(\gamma,\gamma^{\prime}))+|\epsilon|\Big)\mathbb{E}\left[\sum_{k =0}^{N-1}|\bar{x}_{\gamma}(k)|\right] \\
&\ \ \ \ \ \ \ \ \ \ \ \ \ \ \ \ \ +C(L)\mathbb{E}\bigg[\sum_{k =0}^{N-1}\Big(1+|\bar{x}_{\gamma}(k)|^2+|{x}^*_{\gamma}(k)|^2+|{x}^*_{\gamma^{\prime}}(k)|^2+|u^*(k)|^2\Big)I_{R,\xi}\bigg],
\end{split}
\end{align*}
where $I_{R,\xi}=I_{\{|x_{\gamma}^*(k)|\geq R\}}+I_{\{|u^*(k)|\geq R\}}+I_{\{|{x}^*_{\gamma}(k)-{x}^*_{\gamma^{\prime}}(k)|\geq \xi\}}$.
Therefore, a similar argument as Lemma \ref{myw201} together with the estimations \eqref{A3} and \eqref{A7} gives
$$
\begin{aligned}
\begin{split}
\lim\limits_{\epsilon\rightarrow 0}\sup\limits_{\widetilde{d}(\gamma,\gamma^{\prime})\leq \epsilon}\mathbb{E}\Big[|{{K}}_{\gamma,\gamma^{\prime}}|\Big]&\leq  C(L,N,q,d,x_0)\lim\limits_{\epsilon\rightarrow 0}\bigg(|\epsilon|+\sup\limits_{\widetilde{d}(\gamma,\gamma^{\prime})\leq \epsilon}\psi_R(\widetilde{d}(\gamma,\gamma^{\prime}))\\
&\ \ \ +R^{\frac{2-q}{q}}+\xi^{\frac{2-q}{q}}\sup\limits_{\widetilde{d}(\gamma,\gamma^{\prime})\leq \epsilon}{\mathbb{E}\Big[\sum_{k =0}^{N}|{x}^*_{\gamma}(k)-{x}^*_{\gamma^{\prime}}(k)|\Big]^{\frac{q-2}{q}}}\bigg).
\end{split}
\end{aligned}
$$
Going back to \eqref{AA4}, with the help of H{\"o}lder's inequality, we conclude that
$$
\begin{aligned}
\begin{split}
|\bar{y}^u_{\gamma}(0)-\bar{y}^u_{\gamma^{\prime}}(0)| &\leq C(L,N) \bigg\{\mathbb{E}\Big[1+\sum_{k =0}^{N}|{x}^*_{\gamma^{\prime}}(k)|^2+\sum_{k =0}^{N-1}|u^*(k)|^2\Big]^{\frac{1}{2}}\mathbb{E}\Big[\sum_{k =0}^{N}|\bar{\alpha}(k)|^2\Big]^{\frac{1}{2}}+\mathbb{E}\Big[|{{K}}_{\gamma,\gamma^{\prime}}|\Big]\bigg\}.
\end{split}
\end{aligned}
$$
Hence, the required result comes immediately from the continuity of $\gamma\rightarrow\bar{x}_{\gamma}(k)$ proved in step 1.
\end{proof}

\end{document}